\documentclass[11pt]{amsart}
\usepackage{graphicx,amssymb,amstext,amsmath,amsfonts,amsthm}
\usepackage{bbm}
\usepackage{hyperref}
\usepackage{mathtools}
\mathtoolsset{showonlyrefs}
\usepackage{color}
\usepackage{dsfont}
\usepackage{nicefrac}
\usepackage{url}
\usepackage{hyperref}
\usepackage[normalem]{ulem}
\usepackage{anysize}
\usepackage{geometry}
\marginsize{2cm}{2cm}{2cm}{2cm}
\theoremstyle{plain}
\numberwithin{equation}{section}
\newtheorem{lem}{Lemma}[section]
\newtheorem{thm}{Theorem}
\newtheorem{cor}[lem]{Corollary}
\newtheorem{prop}{Proposition}
\newtheorem{hyp}{Hypothesis}
\newtheorem{examp}[lem]{Example}
\newtheorem{rem}[lem]{Remark}
\newcommand{\EE}{\mathbb{E}}
\newcommand{\PP}{\mathbb{P}}
\newcommand{\RR}{\mathbb{R}}
\newcommand{\NN}{\mathbb{N}}
\newcommand{\ind}{\mathbf{1}}
\newcommand{\bB}{\mathcal{B}}
\newcommand{\cC}{\mathcal{C}}
\newcommand{\dD}{\mathcal{D}}

\newcommand{\hH}{\mathcal{H}}
\newcommand{\oO}{\mathcal{O}}
\newcommand{\pP}{\mathcal{P}}
\newcommand{\rR}{\mathcal{R}}
\newcommand{\uU}{\mathcal{U}}
\newcommand{\yY}{\mathcal{Y}}
\newcommand{\zZ}{\mathcal{Z}}
\newcommand{\W}{\mathcal{W}_{p_*}}
\newcommand{\Wp}{\mathcal{W}_{p}}
\newcommand{\fq}{\mathfrak{q}}
\newcommand{\ft}{\mathfrak{t}}
\newcommand{\fT}{\mathfrak{T}}
\newcommand{\cc}{\mathsf{c}}
\newcommand{\Rea}{\mathsf{Re}}
\newcommand{\Ima}{\mathsf{Im}}
\newcommand{\ud}{\mathrm{d}}
\newcommand{\lra}{\longrightarrow}
\newcommand{\e}{\varepsilon}
\newcommand{\ra}{\rightarrow}
\newcommand{\pd}{\partial}
\newcommand{\ti}{\tilde}
\newcommand{\lqq}{\leqslant}
\newcommand{\gqq}{\geqslant}
\newcommand{\<}{\langle}
\renewcommand{\>}{\rangle}
\DeclareMathOperator{\trace}{\ensuremath{trace}}
\title[The cutoff phenomenon in Wasserstein distance]{The cutoff phenomenon in Wasserstein distance for nonlinear stable Langevin systems with small L\'evy noise}
\author{G. Barrera}
\address{University of Helsinki, Department of Mathematical and Statistical Sciences. Exactum in Kumpula Campus. PL 68, Pietari Kalmin katu 5.
Postal Code: 00560. Helsinki, Finland.}
\email{gerardo.barreravargas@helsinki.fi}
\author{M.A. H\"ogele}
\address{Departamento de Matem\'aticas, Facultad de Ciencias, Universidad de los Andes, Bogot\'a, Colombia.}
\email{ma.hoegele@uniandes.edu.co}
\author{J.C. Pardo}
\address{
 CIMAT. Jalisco S/N, Valenciana, CP 36240. Guanajuato, Guanajuato, M\'exico.}
\email{jcpardo@cimat.mx }
\begin{document}
\begin{abstract}  
This article establishes the cutoff phenomenon in the Wasserstein distance 
for systems of nonlinear ordinary differential equations with a  
dissipative stable fixed point subject to small additive Markovian noise.
This result generalizes the results shown in Barrera, H\"ogele, Pardo (EJP2021) in a more restrictive setting of Blumenthal-Getoor index $\alpha>3/2$ to the formulation in Wasserstein distance, which allows to cover the case of general L\'evy processes with some given moment. The main proof techniques are based on the close control of the errors in a version of the 
Hartman-Grobman theorem and the adaptation of the linear theory established in Barrera, H\"ogele, Pardo (JSP2021). In particular, they rely on the precise asymptotics of the nonlinear flow and the nonstandard shift linearity property of the Wasserstein distance, which is established by the authors in (JSP2021). 
Main examples are the nonlinear Fermi-Pasta-Ulam-Tsingou gradient flow and dissipative nonlinear oscillators subject to small (and possibly degenerate) Brownian or arbitrary $\alpha$-stable noise.
\end{abstract}
\maketitle
%\tableofcontents
\section{\textbf{Introduction}} 
\noindent In this paper, we study the asymptotics of the ergodic behavior of the following stochastic differential equation (SDE)
\begin{equation}\label{eq:SDE99}
\ud X^\e_t(x) = -b(X^\e_t(x))\ud t + \e \ud L_t, \quad X^\e_0(x) = x\in \RR^d  
\end{equation}
for small noise intensity $\e>0$, where 
the vector field $b\in \cC^2(\RR^d,\RR^d)$ satisfies  $b(0)=0$ and the following dissipative   condition.
\begin{hyp}[Dissipativity]\label{hyp:dissipa9} There exists a constant $\delta>0$ such that
\begin{equation}\label{eq:cc1}
\<b(x)-b(y),x-y\>\gqq \delta |x-y|^2 \qquad \textrm{ for all } ~x,y\in \RR^d.
\end{equation}
\end{hyp}
\noindent
The noise process $L=(L_t)_{t\gqq 0}$ in \eqref{eq:SDE99} is a L\'evy process with values in $\RR^d$ on a given probability space $(\Omega, \mathcal{F}, \PP)$.
It is well-known that the law of $L$ is characterized by the  triplet $(a,\Sigma, \nu)$, where $a\in \RR^d$,  $\Sigma \in \RR^{d \times d}$ is a non-negative definite matrix and $\nu: \bB(\RR^d) \ra [0, \infty]$
is a locally finite Borel measure satisfying 
\[
\nu(\{0\}) = 0 \qquad \mbox{ and } \qquad \int_{\RR^d} (1\wedge |z|^2) \nu(\ud z) < \infty.
\]
For $\nu=0$ the process $L$ is a multidimensional Brownian motion with drift, while for $a=0$ and $\Sigma=0$ we have a multidimensional pure jump process 
such as compound Poisson processes or $\alpha$-stable processes, in particular, the Cauchy process for $\alpha=1$.
We refer to \cite{Appl,Sa,SITU,IKEDAWATANABE} for further details on L\'evy processes.
Under Hypothesis~\ref{hyp:dissipa9}, it is known that the SDE \eqref{eq:SDE99} has a pathwise unique strong
solution, see for instance Theorem~1.1 in \cite{Majka2020}, here denoted by $X^\e(x):=(X^\e_t(x))_{t\gqq 0}$. Moreover, $X^\e(x)$ is a Markov process and, 
in particular, it satisfies the Feller property see Proposition~2.1 in \cite{Wang2010}.

In order to present the main results of this paper,  we formally introduce the Wasserstein distance of order $p_*$. We assume some finite moment for $L_t$ and hence $X^{\e}_t(x)$ for all $t\gqq 0$.
\begin{hyp}[Finite $p_*$-th moment]\label{hyp:moments9}
There exists $p_*>0$ such that
\[
\int_{|z|>1} |z|^{p_*} \nu(\ud z) < \infty.
\]
\end{hyp}
\noindent
This article shows the cutoff phenomenon for the family of processes $(X^{\e}(x))_{\e>0}$ with respective invariant measures $(\mu^\e)_{\e>0}$ under the Wasserstein distance $\mathcal{W}_{p_*}$ of order $p_*>0$.
For $p_*>1$ we characterize  the following cutoff profile asymptotics 
\begin{equation}\label{eq:repreeps}
\mathcal{W}_{p_*}(\mbox{Law}(X^\e_{\ft_\e+r}(x)),\mu^\e)=\e\cdot C e^{-\fq r}+o(\e)\quad \textrm{ for }\quad \e\to 0,
\end{equation}
where $\ft_\e=\frac{1}{\fq}|\ln(\e)|+\frac{\ell-1}{\fq}\ln(|\ln(\e)|)$ for some explicit positive constants $\fq,\ell,C$ that depend on $x$ in terms of an $\omega$-limit set of the rotational part for the Hartman-Grobman linearization of $X^0(x)$. 

For such processes $(X^{\e}(x))_{\e>0}$ where \eqref{eq:repreeps} fails, we establish the following weaker window cutoff asymptotics

\begin{align*}
\lim\limits_{r\to \infty}\limsup\limits_{\e\to 0}
\frac{\mathcal{W}_{p_*}(\mbox{Law}(X^\e_{\ft_\e+r}(x)),\mu^\e)}{\e}=0\qquad \textrm{and}\qquad
\lim\limits_{r\to -\infty}\liminf\limits_{\e\to 0}
\frac{\mathcal{W}_{p_*}(\mbox{Law}(X^\e_{\ft_\e+r}(x)),\mu^\e)}{\e}=\infty.
\end{align*}
Our results  generalize the results 
in \cite{BHPWA} to the nonlinear vector field and \cite{BHPTV}, \cite{BJ1} and \cite{BP} to the Wasserstein distance which cover second order equations with degenerate noise.
For a detailed introduction on the subject we refer to the aforementioned articles, in particular, see Table~1.1 in \cite{BHPTV}.
There is a particular advantage of studying this  problem under the Wasserstein  distance rather than in  the total variation. While the Wasserstein  distance only requires the existence of moments of $X^{\e}(x)$ of a given order, the total variation distance needs existence  of its density in addition to its regularity.  The latter brings further requirements for  the L\'evy process $L$ which can be quite restrictive, see \cite{BHPTV} for further details. Furthermore the Wasserstein case, at least in case of $X^{\e}(x)$ moments of order $p> 1$,  the cutoff phenomenon  of $(X^{\e}(x))_{\e> 0}$  is  completely determined  by an explicit function (see Theorem~\ref{th:profileabstract9} below),  here called as cutoff profile. On the contrary, in the total variation case the profile function can be very involved and even hard to simulate in examples.  

In \cite{BJ}, the cutoff  phenomenon with respect to the total variation distance covering SDEs of the type (\ref{eq:SDE99})  in the one dimensional case, $L$ being a standard Brownian motion and with general drift coefficient $b$ (satisfying Hypothesis~\ref{hyp:dissipa9}) is studied.  Since scalar systems are gradient systems, there is always  a cutoff profile which can be given explicitly in terms of the Gauss error function.  The follow-up work \cite{BJ1} 
covers the multidimensional case, 
where the picture is considerably richer, 
due to the presence of strong and complicated rotational patterns. 
The authors characterize sharply 
the existence of a cutoff  profile 
in terms of the omega limit sets appearing 
in the long-term behavior of the 
matrix exponential function $e^{-\mathcal{Q} t}x$ in Lemma~B.2 in \cite{BJ1}, 
which plays an analogous role in this article.  
The paper \cite{BP} is the first attempt to study the cutoff phenomenon for such models with jumps.
More precisely, \cite{BP} covers the cutoff phenomenon with respect to the total variation distance of the generalized Ornstein-Uhlenbeck processes. The previous process satisfies an SDE of the form (\ref{eq:SDE99}) with $L$ being a L\'evy process and   $b(x)=\mathcal{Q} x$, where $\mathcal{Q}$ is a square real matrix whose eigenvalues have positive real parts. The proof methods are based on concise  Fourier inversion techniques. 
Due to the aforementioned  regularity inherited by 
the total variation, the results in \cite{BP} are given 
under the hypothesis of continuous densities of 
the marginals, which to date is mathematically not characterized 
in simple terms. 
The cutoff profile function in \cite{BP} is given in terms 
of the L\'evy-Ornstein-Uhlenbeck limiting measure for $\e=1$ and 
measured in the total variation distance. 
Such profile functions are theoretically highly  insightful, but almost impossible to calculate and simulate in examples. The characterization of the existence of 
a cutoff-profile remains analogously to \cite{BJ1} 
in abstract terms of the behavior of the mentioned profile 
function on a suitably defined omega limit set.  The Wasserstein case is treated in  \cite{BHPWA} where, contrary to the total variation case,  it is noted that the profile function takes an explicit and simple shape. Finally, \cite{BHPTV} treats the cutoff phenomenon with respect to the total variation distance for (\ref{eq:SDE99}) with $b$ satisfying Hypothesis~\ref{hyp:dissipa9} and  driven by a L\'evy  process in the rather restrictive class of strongly locally layered stable processes (see Definition 1.4 in \cite{BHPTV}). 

In this article we combine a nonlinear version of the Wasserstein estimates of \cite{BHPWA},  
with the Freidlin-Wentzell first order approximation of \eqref{eq:SDE99} in the spirit of \cite{BHPTV} and the fine properties of the Wasserstein distance given in Lemma~\ref{lem:properties}, in particular, the non-standard shift linearity  of Lemma~\ref{lem:properties}.d).

The manuscript is organized in four parts.
After the exposition of the setting and the presentation of the main results in Section~\ref{sec:settingmainresults}, we illustrate  our findings for the nonlinear Fermi-Pasta-Ulam-Tsingou gradient system and a class of nonlinear oscillators in Section~\ref{sec:examples}. The main steps of the proof of the cutoff phenomenon are given in Section~\ref{sec:proofs} while the auxiliary technical such as exponential ergodicity in Wasserstein distance, the coupling between the original nonlinear system and the Freidlin-Wentzell linearization results are given in the appendix.

\section{Setting and main results}\label{sec:settingmainresults}
\subsection{Fine properties of the Wasserstein distance} 
\noindent For any two probability distributions $\mu_1$ and $\mu_2$ on $\RR^d$
with finite $p_*$-th moment for some $p_*>0$, we 
define the Wasserstein $p_*$-distance between them as follows
\begin{equation*}\label{def:wasserstein}
\W(\mu_1,\mu_2)=
\inf_{\Pi} \left(\int_{\RR^d\times \RR^d}|u-v|^{p_*}\Pi(\ud u,\ud v)\right)^{1\wedge (1/p_*)},
\end{equation*}
where the infimum is taken over all 
couplings (joint distributions on $\RR^d\times \RR^d$) $\Pi$ with marginals $\mu_1$ and $\mu_2$. 
We refer to \cite{PZ20,Villani09} and references therein for more details. 
For convenience of notation we do not distinguish a random variable $U$ and its law $\PP_U$ as an argument of $\W$. That is, for random variables $U_1$, $U_2$ and probability measure $\mu$ we write $\W(U_1, U_2)$ instead of $\W(\PP_{U_1}, \PP_{U_2})$, $\W(U_1, \mu)$ instead of $\W(\PP_{U_1}, \mu)$ etc. 
The next result establishes properties of the Wasserstein distance which turn out to be important for our arguments.
\begin{lem}[Properties of $\W$]\label{lem:properties}
For $p_*>0$,
$u_1,u_2\in \RR^d$, $c\in \RR$ and $U_1$ and $U_2$ being random vectors in $\RR^d$ 
with finite $p_*$-th moment we have the following: 
\begin{itemize}
\item[a)] The Wasserstein distance $\W$ is a metric.
\item[b)] Translation invariance: 
$\W(u_1+U_1,u_2+U_2)=\W(u_1-u_2+U_1,U_2)$.
\item[c)] Homogeneity: 
\[
\W(c\cdot U_1,c\cdot U_2)=
\begin{cases}
|c|\;\W(U_1,U_2)&\textrm{ for } p_*\in [1,\infty),\\
|c|^{p_*}\;\W(U_1,U_2)&\textrm{ for } p_*\in (0,1).
\end{cases}
\]
\item[d)] Shift linearity: For $p_*\gqq 1$ it follows 
\begin{equation}\label{eq:shitflinearity}
\W(u_1+U_1,U_1)=|u_1|.
\end{equation}
For $p_*\in (0,1)$ we have 
\begin{equation}\label{ec:cotaabajop01}
\max\{|u_1|^{p_*}-2\EE[|U_1|^{p_*}],0\}\lqq 
\W(u_1+U_1,U_1)\lqq |u_1|^{p_*}.
\end{equation}
\item[e)] Domination: For any given coupling $\ti \Pi$ between $U_1$ and $U_2$ it follows  
\[
\W(U_1, U_2) \lqq \Big(\int_{\RR^d\times \RR^d} |v_1-v_2|^{p_*} \ti \Pi(\ud v_1,\ud v_2)\Big)^{1\wedge (1/p_*)}.
\]
\item[f)] Characterization: Let $(U_n)_{n\in \mathbb{N}}$ be a sequence of random vectors with finite $p_*$-th moments 
and $U$ a random vector with finite $p_*$-th moment. Then the following statements are equivalent: 
\begin{enumerate}
 \item $\W(U_n, U) \ra  0$ as $n\ra  \infty$. 
 \item $U_n \stackrel{d}{\lra} U$ as $n \ra  \infty$ and $\EE[|U_n|^{p_*}] \ra  \EE[|U|^{p_*}]$ as $n\ra  \infty$. 
\end{enumerate}
\end{itemize}
\end{lem}
\noindent For $p_*\in(0,1)$ equality \eqref{eq:shitflinearity} is false in general, see Remark~2.4 in \cite{BHPWA}.
The proof of the previous lemma is given in Lemma~2.2 in \cite{BHPWA}. 

The following result yields the existence of a unique invariant distribution for \eqref{eq:SDE99} under Hypotheses \ref{hyp:dissipa9} and \ref{hyp:moments9}. Moreover, under the Wasserstein distance, the strong solution of \eqref{eq:SDE99} is exponentially ergodic. 
\begin{prop}[Existence of a unique invariant distribution]
Under Hypothesis~\ref{hyp:dissipa9} for $p_*>0$ and Hypothesis~\ref{hyp:moments9} there exists a unique invariant probability measure $\mu^\e$ such that 
\begin{equation}
\W(X^\e_t(x),\mu^{\e})\lqq e^{-({1\wedge p_*})
\delta t}
\left(
|x|^{1\wedge p_*}+\int_{\RR^d}|y|^{1\wedge p_*}\mu^{\e}(\ud y)
\right).
\end{equation}
\end{prop}
\noindent The proof is given in Appendix~\ref{AP1:existencia}.
\subsection{Hartman-Grobman asymptotics} 
\noindent The zeroth-order approximation of a smooth dynamical systems on a finite time  horizon $[0,T]$ subject to small perturbations 
is given by the deterministic system, that is, $(X^0_t(x))_{t\in [0,T]}$. 
Our main results treat small asymptotics close to the stable state $0$ which translates to meaningful time scales  $t_\e\ra  \infty$, as $\e\ra 0$, in 
Theorem~\ref{th:linearwhite} and Theorem~\ref{th:profileabstract9}.
Before we state our main result, we first provide the long-time asymptotics of $X^0_t(x)$ in terms of the spectral decomposition of the solution $t\mapsto e^{-Db(0)t}x^*$ of the respective  linear system for some $x^*$ in a small neighbourhood of the origin.
\begin{lem}[Asymptotic Hartman-Grobman]\label{lem:jara29}\hfill\\ 
Assume Hypothesis~\ref{hyp:dissipa9}. 
Then for  any $x\in \mathbb{R}^{d}\setminus\{0\}$  there exist: 
\begin{enumerate}
\item[(i)] positive constants $ \fq^x, \tau^x,\ell^x, m^x$ with $\ell^x,m^x\in \{1,\ldots,d\}$,
\item[(ii)] angular velocities $\theta^{x}_{1},\dots,\theta^x_{m^x}\in \RR$,
where all $\theta^x_k \neq 0$ come in pairs
$(\theta^x_{j_*},\theta^x_{j_*+1})=(\theta^x_{j_*}, -\theta^x_{j_*})$,
\item[(iii)] linearly independent vectors $v_1^x,\dots,v_{m_x}^x$ in $\mathbb{C}^d$ which are complex conjugate 
$(v^x_{j_*},v^x_{j_*+1})=(v^x_{j_*}, \bar{v}^x_{j_*})$  whenever $(\theta^x_{j_*},\theta^x_{j_*+1})=(\theta^x_{j_*}, -\theta^x_{j_*})$,
\end{enumerate}
such that
\begin{equation}\label{eq:hartmangrobman}
\lim_{t \ra \infty} 
\left| \frac{e^{\fq^x t}}{t^{\ell^x-1}} X^0_{t+\tau^x}(x) - \sum_{k=1}^{m^x} e^{i\theta^x_k t}v^x_k \right|=0.
\end{equation}
Moreover,
\begin{equation}\label{eq:belowabovebound}
0<\liminf_{t\rightarrow \infty}\left|\sum_{k=1}^{m^x} e^{i  t\theta^x_k} v^x_k\right|
\lqq 
\limsup_{t\rightarrow \infty}\left|\sum_{k=1}^{m^x} e^{i  t\theta^x_k} v^x_k\right|\lqq 
\sum_{k=1}^{m^x}  |v^x_k|.
\end{equation}
\end{lem}
\noindent The formal proof of the previous lemma is given in Lemma~B.2 in Appendix~B of \cite{BJ1}.
\begin{rem}\label{rem:convention}\hfill
\begin{enumerate}
 \item Convention: Note that 
$\theta^x_k=0$ is true for at most one index $k\in \{1,\ldots, m^x\}$. 
If such an index shows up in $\theta^x_{1},\ldots, \theta^x_{m^x}$
we adopt the convention that $\theta^x_1=0$ and $v_1^x\in \RR^d$, 
and hence $m^x=2n+1$ for some $n\in \mathbb{N}_0$.
Otherwise, $m^x=2n$ for some $n\in \NN_0$ and we eliminate $\theta^x_1$ and count the angular velocities as follows $\theta^x_2,\ldots, \theta^x_{2n+1}$.
 \item Note that the linearly independent complex vectors 
$v_1^x,\dots,v_{m_x}^x$ in $\mathbb{C}^d$ not only depend on $x$ but also crucially on the dissipation time $\tau^x$ of the deterministic system to a Hartman-Grobman domain of conjugacy $U$. We stress that $\tau^x$ is not unique since $X^0_{t+\tau^x}(x) \in U$ for all $t\gqq 0$.
\item A word about the parameters $\ell^x$, $\fq^x$ and $m^x$ in Lemma~\ref{lem:jara29}. By the Hartman-Grobman theorem there are open sets $0\in U, V\subset \RR^d$ and a homeomorphism $H: U\ra V$ with $H(0) = 0$ satisfying for all $u\in U$ and $t\gqq 0$ 
\begin{equation}\label{e: conj1}
H(X^0_t(u)) = e^{- Db(0) t}H(u).    
\end{equation}  
In fact, by Hypothesis~\ref{hyp:dissipa9} we have that $H$ is a $\cC^1$-diffeomorphism, see the original paper~\cite{Ha60}  or Theorem(Hartman), Sec. 2.8, p.127, \cite{Pe01}. In \cite{Ha60} it is shown that $H$ can be chosen to be 
\[
H(x) = x + o(|x|)_{|x| \ra 0}. 
\]
Let $\ti u = X^0_{\tau^x}(x) \in U$. 
With the help of a linear coordinate change $W$ 
we obtain the Jordan normal form $Db(0) = W^{-1} J(Db(0)) W$ and 
(using the linearity of the semigroup) 
\[
H(X^0_{t+\tau^x}(x)) = W^{-1} e^{- J(Db(0)) t} (W H(\ti u)).    
\]
We denote $\ti w = W H(\ti u)$. Now, the parameters $\ell^x$, $\fq^x$ and $m^x$ 
are given as follows. Consider the sequence of generalized eigenspaces $H_{j}$ of $J(Db(0))$ such that 
\[
\RR^d = H_1\oplus \dots \oplus H_{k_*}.  
\]
By construction, $\ti w \in G(\ti w) := \mbox{span}(\{H_k~|~\mbox{where }1\lqq k\lqq k_*: ~\mbox{proj}(\ti w, H_k)\neq 0\})$. Note that $G(\ti w)$ is unique. We consider the restriction 
\[
\ti J(\ti w):= J(Db(0))\big|_{G(\ti w)}.
\]
Now, $\fq^x$ is the smallest real part of the spectrum of $\ti J(\ti w)$,  
$\ell^x$ is the dimension of the largest Jordan block of $\ti J(\ti w)$ which has the real part $\fq^x$ and $m^x$ is the number of Jordan blocks associated to $\fq^x$ and $\ell^x$. Note that in case of a non real eigenvalue with real part $\fq^x$ and Jordan block size $\ell^x$, we have $m^x\gqq 2$. For an extensive numerical example for a linear chain of oscillators we refer to Section 4.3.2 in \cite{BHPWA}.
\end{enumerate}
\end{rem}

\subsection{Main results} 
\noindent Our first main result establishes 
$\infty/0$ collapse of the Wasserstein distance between the law of the current state $X^{\e}_t(x)$ and the dynamical equilibrium $\mu^\e$ along the critical time scale $\ft^x_\e$ given in \eqref{eq:timescale} under  mild conditions.
\begin{thm}[Window cutoff]
\label{th:linearwhite}
Let $b$ satisfy Hypothesis~\ref{hyp:dissipa9} and $\nu$ satisfy 
Hypothesis~\ref{hyp:moments9} for some  $p_*>0$. Fix $x\in \RR^d\setminus\{0\}$ and consider the notation in the asymptotic Hartman-Grobman representation $\fq^x>0$, $\ell^x , m^x \in  \{1,\ldots, d\}$, $\theta^x_1,\dots,\theta^x_{m^x} \in [0,2\pi)$, $v^x_1,\dots,v^x_{m^x} \in  \mathbb{C}^d$ and $\tau^x>0$ of Lemma~\ref{lem:jara29}. 

Then the
family of processes $(X^{\e}(x))_{\e>0}$ exhibits a window cutoff phenomenon
on the time scale  
\begin{equation}\label{eq:timescale}
\ft^x_\e=\frac{1}{\fq^x}|\ln(\e)|+\frac{\ell^x-1}{\fq^x}\ln(|\ln(\e)|)
\end{equation}
and for all asymptotically constant window sizes $w_\e$, that is, $w_\e\ra  w>0$ as $\e \ra 0$, in the following sense. For all $0<p< p^*$ we have 
\begin{equation}\label{eq:infimosupremo}
\lim_{r\ra \infty}\limsup_{\e\ra 0}
\frac{\Wp(X^\e_{\ft^x_\e+r\cdot w_\e}(x),\mu^\e)}{\e^{{1\wedge p}}}=0
\qquad \mbox{ and }\qquad 
\lim_{r\ra -\infty}\liminf_{\e\ra 0}
\frac{\Wp(X^\e_{\ft^x_\e+r\cdot w_\e}(x),\mu^\e)}{\e^{1\wedge p}}=\infty.
\end{equation}
 \end{thm}
\noindent The second main result provides two characterizations for the proper limits ($\e \ra 0$) of the expressions in~\eqref{eq:infimosupremo} for any fixed $r\in \RR$. 
That is to say, we characterize under which conditions the asymptotics \eqref{eq:repreeps} is satisfied. In addition, it yields  the precise shape of the limit which turn out to be a simple exponential function for $p\in [1,p_*)$.
\begin{thm}[\textbf{Dynamical  profile cutoff characterization for $p_*>0$}]\label{th:profileabstract9} \hfill\\
Let the assumptions (and the notation) of Theorem~\ref{th:linearwhite}
be valid for some $p_*>0$. Consider the unique strong solution $(\oO_t)_{t\gqq 0}$ of the linear system 
\begin{equation}\label{e:OU}
\ud \oO_t=-Db(0)\oO_t+\ud L_t,
\end{equation}
where $\oO_\infty$ is the unique invariant probability distribution of \eqref{e:OU}. 
\begin{enumerate}
\item Then for any $0<p< p_*$ the following statements are equivalent.
\begin{enumerate}
\item[i)] For any $\lambda>0$, the function
$\omega(x)\ni u\mapsto
\Wp(\lambda u+\oO_\infty,\oO_\infty)$
is constant, where
\begin{equation*}
\omega(x):=
\Big\{
\textrm{accumulation points of }
\sum_{k=1}^m e^{i t \theta^x_k} v^x_k
\textrm{ as } t\ra \infty
\Big\}. 
\end{equation*}
\item[ii)] 
The family of processes $(X^{\e}(x))_{\e>0}$ exhibits
a  profile cutoff
 for any $0< p< p_*$ as follows
\[
\lim_{\e\ra 0}
\frac{\Wp(X^\e_{\ft^x_\e+r\cdot w_\e}(x),\mu^\e)}{\e^{1\wedge p}}=
\pP^{x}_{p}(r) \quad \textrm{ for any }
r\in \RR,
\]
where 
\begin{equation}\label{eq:abstractperfil}
\pP^{x}_{p}(r):=\Wp\Big(\kappa^x(r)\cdot v+
\oO_\infty,\oO_\infty\Big) \qquad \mbox{ for any }v\in \omega(x)
\end{equation}
and 
\begin{align*}\kappa^x(r)= \frac{e^{-\fq^x r\cdot w}}{e^{\fq^x \tau^x}(\fq^x)^{\ell^x-1}}.\\\end{align*}
\end{enumerate}
\item For $p_*> 1$ and $p\in [1,p_*)$ 
the profile has the shape
\[
\pP^{x}_{p}(r)=\kappa^x(r)\cdot |v|\quad \textrm{ for all }v\in \omega(x)
\]
if and only if $\omega(x)$ is contained in a sphere in $\RR^d$ with respect to the Euclidean norm.\\
\item  
We recall the convention of Remark~\ref{rem:convention}.
Let $p_*> 1$ and $p\in [1,p_*)$. 
If the angles $\theta^x_{2},\ldots, \theta^x_{2n}$ satisfy the following non-resonance condition 
\begin{equation}\label{eq:non-resonance}
h_1\theta_2 + \cdots + h_n \theta_{2n} \in 2 \pi \cdot \mathbb{Z} \qquad \mbox{ for all }(h_1, \ldots, h_n)\in \mathbb{Z}^n\setminus \{0\}, 
\end{equation} 
then the statements i) and ii) in item (1) 
are equivalent to the following normal growth condition of the asymptotic Hartman-Grobman linearization: The family of limiting vectors  
\[(v_1^x,\Rea\,v^x_2,\Ima\, v^x_2,\ldots,\Rea\, v^x_{2n},\Ima\, v^x_{2n})\]
 is orthogonal in $\RR^d$ and satisfies
\begin{align*}|\Rea\, v^x_{2k}|=|\Ima\, v^x_{2k}|\qquad \mbox{ for all }\quad k=1,\ldots,n.\\\end{align*}
\end{enumerate}
\end{thm}

\begin{rem}
We stress that $\oO_\infty = \lim_{t\ra \infty} \oO_t$ in $\mathcal{W}_{p_*}$ and due to Hypothesis~\ref{hyp:dissipa9} 
(in combination with Hypothesis~\ref{hyp:moments9}) the distribution of $\oO_\infty$ does not depend on any deterministic initial condition of \eqref{e:OU}. 
\end{rem}

\noindent Due to its relevance as physical observables, we formulate the corresponding window cutoff result for the respective moments.
\begin{cor}[Moments cutoff]\label{cor:momentscutoffOUP}
Let the assumptions (and the notation) of Theorem~\ref{th:linearwhite}
be valid for some $p_*>0$.
Then for any $0< p<p_*$ it follows\\
\begin{align*}
\lim_{r\ra \infty}
\liminf_{\e\ra 0}\frac{\EE[|X^{\e}_{\ft_\e^x+r\cdot w_\e}(x)|^{p}]}{\e^{ p}} &=\lim_{r\ra \infty}
\limsup_{\e\ra 0}\frac{\EE[|X^{\e}_{\ft_\e^x+r\cdot w_\e}(x)|^{p}]}{\e^{ p}}=
\EE[|\oO_{\infty}|^{p}],\\[2mm]
\lim_{r\ra -\infty}
\liminf_{\e\ra 0}\frac{\EE[|X^{\e}_{\ft_\e^x+r\cdot w_\e}(x)|^{p}]}{\e^{p}}&=\lim_{r\ra -\infty}
\limsup_{\e\ra 0}\frac{\EE[|X^{\e}_{\ft_\e^x+r\cdot w_\e}(x)|^{p}]}{\e^{ p}}=\infty.\\
\end{align*}
\end{cor}

\section{\textbf{Examples}}\label{sec:examples}
In this section we present two examples
which illustrate the applicability of Theorem~\ref{th:linearwhite} and 
Theorem~\ref{th:profileabstract9}
to nonlinear dynamics with degenerate noise.
\begin{examp}[The Fermi-Pasta-Ulam-Tsingou potential]
We consider the nonlinear Langevin gradient system 
\begin{equation}\label{eq:quarticp}
\ud X^\e_t = - \nabla \uU(X^\e_t)\ud t + \e \ud L_t 
\end{equation}
for the strongly convex quartic Fermi-Pasta-Ulam-Tsingou potential $\uU(x) = 
\frac{1}{2} |x|^2 + \frac{1}{4}|x|^4$, $x\in \RR^d$ subject to degenerate noise $\ud L_t$. 
For any L\'evy process $L$ 
satisfying Hypothesis~\ref{hyp:moments9} for some $p_*>0$ the system 
\eqref{eq:quarticp} exhibits 
a profile cutoff due to Theorem~\ref{th:profileabstract9} where 
the cutoff time is given by
$\ft_\e^x = |\ln(\e)|$. 
For $p_*>1$ and any $p\in [1,p_*)$
 the profile function in $\Wp$ 
is always  of the following exponential shape
\begin{equation}\label{e:shapeprofile}
\pP^{x}_{p}(r)=e^{-wr- \tau^x}\Big|\sum_{k=1}^m  v^x_k\Big|,
\end{equation}
where  $\tau^{x}:=\min\{t\gqq 0:|X^0_t(x)|\lqq R_0/2\}$ and $R_0$ being an small radius inside of which Hartman-Grobman conjugation is valid. Note that $\tau^x$ can be replaced by any upper bound of $\tau^x$ such as for instance $(\nicefrac{1}{\delta})\ln(2|x|/R_0)$ given by Hypothesis~\ref{hyp:dissipa9}.

In particular, the profile cutoff \eqref{e:shapeprofile} is valid for $L=L^\alpha$ being an (possibly degenerate) $\alpha$-stable process with index $\alpha\in (1,2]$. Note that for the limiting case of a possibly degenerate Cauchy process ($\alpha=1$) and in fact of any 
$L^\alpha$ with index
$\alpha\in (0,1)$,
Theorem~\ref{th:profileabstract9} also yields a profile cutoff. However, the profile function remains not explicit. This is due to the absence of a finite first moment
and the lack of the shift linearity \eqref{ec:cotaabajop01}.
In other words, the profile  function is given in \eqref{eq:abstractperfil}
for $p\in (0,\alpha)$ and up to our knowledge unknown how to simplify further.
Note that the case of $\alpha\in (0,3/2]$ is new and is 
not covered in \cite{BHPTV}.
\end{examp}
\begin{examp}[Nonlinear non-gradient with degenerate noise]\label{ex:oscillator}
For $F,\hH\in \mathcal{C}^2(\RR^2,\RR)$
we consider the following perturbed simple harmonic oscillator 
with unit
angular frequency
given in Section~4 of \cite{Tudoran}  subject to a small noise perturbation
\begin{equation*}
\ud \left(
\begin{matrix}
 X^{\e,1}_t \\
 X^{\e,2}_t 
\end{matrix}
\right)=-
\left(
\begin{array}{c}
X^{\e,2}_t \,F(X^{\e,1}_t,X^{\e,2}_t)-\partial_1 \hH(X^{\e,1}_t,X^{\e,2}_t) \\
-X^{\e,1}_t \,F(X^{\e,1}_t,X^{\e,2}_t)-\partial_2 \hH(X^{\e,1}_t,X^{\e,2}_t)
\end{array}
\right) \ud t+ \e \ud\left(
\begin{matrix}
0 \\
\mathcal{L}_t 
\end{matrix}
\right),
\end{equation*}
where 
$\mathcal{L}=(\mathcal{L}_t)_{t\gqq 0}$ is a one dimensional L\'evy process with finite $p_*$-th moments.
The Jacobian matrix $Jb(v_1,v_2)$  at $(v_1,v_2)$ of the respective vector field $b:\RR^2\ra \mathbb{R}^2$ is given by
\begin{align*}
\left(
\begin{matrix}
v_2\partial_1 F(v_1,v_2)-\partial_{11}\hH(v_1,v_2) & F(v_1,v_2)+v_2\partial_2 F(v_1,v_2)-\partial_{12}\hH(v_1,v_2)\\
-F(v_1,v_2)-v_1\partial_{1}F(v_1,v_2)-\partial_{12}\hH(v_1,v_2) & -v_1\partial_{2}F(v_1,v_2)- \partial_{22}\hH(v_1,v_2)
\end{matrix}
\right).
\end{align*}
It is enough to prove the existence of a positive constant $\delta$ such that
for any $u_1,u_2,v_1,v_2\in \RR$ it follows
\begin{align}\label{eq:condition1}
(u_1,u_2) Jb(v_1,v_2)(u_1,u_2)^* &=(v_2\partial_1 F(v_1,v_2)-\partial_{11}\hH(v_1,v_2))u^2_1+(-v_1\partial_2 F(v_1,v_2)-\partial_{22}\hH(v_1,v_2))u^2_2 \nonumber\\
&\qquad+(v_2\partial_2 F(v_1,v_2)-v_1\partial_{1}F(v_1,v_2)-2\partial_{12}\hH(v_1,v_2))u_1u_2 \nonumber\\
&\gqq \delta (u^2_1+u^2_2).
\end{align}
For instance, for a nonlinear perturbation of a linear oscillator, that is,
$F(v_1,v_2)=\eta$ for some $\eta>0$, the preceding condition reads
\[
-\Big(\partial_{11}\hH(v_1,v_2)u^2_1+\partial_{22}\hH(v_1,v_2)u^2_2+2\partial_{12}\hH(v_1,v_2)u_1u_2\Big)\gqq \delta (u^2_1+u^2_2).
\]
For $\mathcal{L}$ satisfying Hypothesis~\ref{hyp:moments9} with $p_*$, and $F$, $\hH$ fulfilling \eqref{eq:condition1} Theorem~\ref{th:linearwhite} implies  window cutoff for any initial condition $(X^{\e,1}_0,X^{\e,2}_0)=x\in \RR^2\setminus\{0\}$ and 
any $p\in (0,p_*)$.
The cutoff time  is given by 
\[
\ft^x_\e = \frac{1}{\fq^x} |\ln(\e)|+\frac{\ell^x-1}{\fq^x}\ln(|\ln(\e)|).  
\]
Note that this result is
new even in the Brownian case since the results of 
  \cite{BHPTV} and \cite{BJ1} 
are stated for the total variation distance which requires regularity on the transition probabilities given in the setting of non-degenerate noise. In our case, the Wasserstein distance circumvents this difficulty by the
continuity of $\Wp(x+X,X)$ for any $X\in L^{p}$ as $|x|\ra 0$ and $|x|\ra \infty$, 
while for total variation distance it requires absolutely continuity on the distribution of $X$. 
We refer to \cite{BHPTV}, Lemma~1.17 in Subsection~1.3.5, for an example where the continuity of the total variation distance under shifts is not valid.  
 
In the sequel, we characterize the existence of  a profile cutoff under \eqref{eq:condition1}
in terms of the linearization at the stable state $(0,0)$.
Let  $a:=-\partial^2_{11}\hH(0,0)$ $b:=-\partial^2_{22}\hH(0,0)$, $c:=-\partial_{12}\hH(0,0)$
and $\eta_0:=-F(0,0)$. Then
\begin{align*}\label{eq:Jac}
Jb(0,0)=
\left(
\begin{matrix}
a & -\eta_0+c\\
\eta_0+c &  b
\end{matrix}
\right).
\end{align*}
Note that $\eta_0=c$ implies that the eigenvalues of $Jb(0,0)$ are the  numbers $a$ and $b$ which are positive 
and hence by Theorem~\ref{th:profileabstract9} profile cutoff is valid.
In the sequel we assume $\eta_0 \neq c$. Then the eigenvalues of $Jb(0,0)$ are given by
\[
\lambda_{\pm}:=\frac{(a+b)\pm \sqrt{\Delta}}{2},\quad \Delta:=(a-b)^2+4(c^2-\eta^2_0),
\]
with corresponding eigenvectors
\[
v_{\pm}:=
\left(1,-\frac{a-b\mp\sqrt{\Delta}}{2(-\eta_0+c)}\right).
\]
In addition, 
\[
\mathsf{Re}(v_{\pm})=\begin{cases}
\left(1,-\frac{a-b\mp\sqrt{\Delta}}{2(-\eta_0+c)}\right)  & \textrm{if } \Delta\gqq 0,\\[2mm]
\left(1,-\frac{a-b}{2(-\eta_0+c)}\right)   & \textrm{if } \Delta< 0,
\end{cases}\qquad \mathsf{ and } \qquad 
\mathsf{Im}(v_{\pm})=\begin{cases}
\left(0,0\right)  & \textrm{if } \Delta\gqq 0,\\[2mm]
\pm\left(0, \frac{ \sqrt{|\Delta|} }{2(-\eta_0+c)}\right)   & \textrm{if } \Delta< 0.
\end{cases}
\]
For $\Delta\gqq 0$
Theorem~\ref{th:profileabstract9} yields a profile cutoff phenomenon.
For $\Delta<0$ Theorem~\ref{th:linearwhite} implies the weaker  window cutoff phenomenon, however, by part (3) of Theorem~\ref{th:profileabstract9}
the stronger profile cutoff for $p_*>1$ and $p\in [1,p_*)$ is valid
if and only if
\[
|\mathsf{Re}(v_{+})|^2=|\mathsf{Im}(v_{+})|^2 \textrm{ and } \<\mathsf{Re}(v_{+}),\mathsf{Im}(v_{+})\>=0
\]
which is equivalent to special case $a=b$ and $c=0$. In other words, 
$e^{-Jb(0,0)t}=e^{-at}R(\theta t)$, where $R(\theta t)$ is an orthogonal $2\times 2$ matrix with angle 
$\theta t$.
\end{examp}
\begin{rem}[A word about the linear dynamics]
In \cite{BHPWA} the authors study \eqref{eq:SDE99} for the linear vector field $b(x)=\mathcal{Q}x$ for any  Hurwitz stable matrix $-\mathcal{Q}$, that is, $\mathsf{Re}(\lambda)<0$ for any eigenvalue $\lambda$ of $-\mathcal{Q}$. 
Under these assumptions, the results of Theorem~\ref{th:linearwhite} and 
Theorem~\ref{th:profileabstract9} are obtained.

It is not hard to see that Hypothesis~\ref{hyp:dissipa9} implies 
$\mathsf{Re}(\lambda)\lqq -\delta$ for any eigenvalue $\lambda$ of $-\mathcal{Q}$ and hence Hurwitz stability. However, the dissipativity condition \eqref{eq:cc1}  which is assumed in order to control the nonlinear vector field, is strictly stronger than Hurwitz stability. For instance, the vector field $b:\RR^2 \ra \RR^2$ given by $b(x)=\mathcal{Q}x$ with 
\[
-\mathcal{Q}=
\left(
\begin{matrix}
0 & -1\\
\lambda & \lambda
\end{matrix}
\right) 
\textrm{ with } \lambda \in (0,1/2)
\]
has eigenvalues with real part $-\lambda/2<0$, but
 it
does not satisfy Hypothesis~\ref{hyp:dissipa9}. Note that the dissipativity condition \eqref{eq:cc1} is not even satisfied locally in a neighborhood of the origin. 
\end{rem}

\section{\textbf{Proofs of the main results}}\label{sec:proofs}
\subsection{\textbf{The first order approximation}}
We define the Freidlin-Wentzell first order approximation given by
\begin{equation}\label{def:Yprocess9}
Y^{\e}_t(x)=X^{0}_t(x)+\e \yY^{x}_t, \qquad t\gqq 0,
\end{equation}
where
$(\yY^{x}_t)_{t\gqq 0}$ is the unique strong solution of the linear inhomogeneous SDE
\begin{align}\label{eq: Yxt9}
\left\{
\begin{array}{r@{\;=\;}l}
\ud \yY^x_t & -Db(X^{0}_t(x)) \yY^{x}_t\ud t+\ud L_t \quad \textrm{ for any } t\gqq 0,\\
\yY^{x}_0 & 0.
\end{array}
\right.
\end{align}
In \cite{BHPTV}, Lemma~C.4 in Section~C.4 it is shown that $Y^{\e}_t(x)$ converges in total variation distance to  a unique limiting distribution $\mu^\e_*$ as $t\ra \infty$. Moreover, it is shown there that
$\mu^\e_*\stackrel{d}=\e \oO_\infty$, where $\oO_\infty$ 
is the unique invariant probability distribution of the homogeneous Ornstein-Uhlenbeck dynamics
\begin{equation}
\ud \oO_t=-Db(0)\oO_t+\ud L_t.
\end{equation}
In the sequel we reduce the nonlinear ergodic convergence of $X^\e_t(x)$ to the ergodic convergence of the Freidlin-Wentzell linearization $Y^\e_t(x)$ in \eqref{eq:pivot} up to error terms. For any $0<p\lqq p_*$, by the triangle inequality it follows that
\[
\Wp(X^\e_t(x),\mu^\e)\lqq \Wp(X^\e_t(x),Y^\e_t(x))+\Wp(Y^\e_t(x),\mu^\e_*)+\Wp(\mu^\e_*,\mu^\e)
\]
for any $t\gqq 0$, $x\in \mathbb{R}^d$.
Analogously we estimate
\[
\Wp(Y^\e_t(x),\mu^\e_*)\lqq \Wp(Y^\e_t(x),X^\e_t(x))+\Wp(X^\e_t(x),\mu^\e)+\Wp(\mu^\e,\mu^\e_*).
\]
Combining the preceding inequalities we obtain the linear approximation
\begin{equation}\label{eq:pivot}
\left|\Wp(X^\e_t(x),\mu^\e)-\Wp(Y^\e_t(x),\mu^\e_*) \right|\lqq 
 \Wp(X^\e_t(x),Y^\e_t(x))+\Wp(\mu^\e,\mu^\e_*)
\end{equation}
for any $t\gqq 0$, $x\in \mathbb{R}^d$.
In Proposition~\ref{cl:linearization} given in Appendix~\ref{app:B2} we show that for any $t_\e=O(|\ln(\e)|)$ and $0< p < p_*$ the following limit holds
\begin{equation}\label{prop:aproxlineal}
\lim\limits_{\e\ra 0}\frac{\Wp(X^\e_{t_\e}(x),Y^\e_{t_\e}(x))}{\e^{1\wedge p}}=0.
\end{equation}
Moreover, in Lemma~\ref{cl:limit} we show that for $0< p < p_*$
\begin{equation}\label{lem:limites}
\lim_{\e \ra 0}\frac{\Wp(\mu^\e_*,\mu^\e)}{\e^{1\wedge p}}=0.
\end{equation}
\subsection{\textbf{Derivation of the cutoff  phenomenon}}
In the sequel, we analyze the asymptotic behavior of 
$\Wp(Y^\e_t(x), \mu^\e_*)\cdot \e^{-(1\wedge p)}$ from which we recognize the cutoff of the Freidlin-Wentzell linearization $Y^\e_t(x)$. 
By the triangle inequality, translation invariance, homogeneity and shift linearity given in Lemma~\ref{lem:properties} we obtain for $0< p\lqq p_*$
\begin{align*}
\Wp(Y^\e_t(x), \mu^\e_*) 
&= \Wp(X^0_t(x) + \e \yY^x_t,  \e \oO_\infty) \\
&\lqq \Wp(X^0_t(x) + \e \yY^x_t,  X^0_t(x) + \e \oO_\infty) + \Wp(X^0_t(x) + \e \oO_\infty, \e \oO_\infty)\\
&= \e^{1\wedge p}\cdot \Wp(\yY^x_t, \oO_\infty) +\e^{1\wedge p}\cdot \Wp(\e^{-1} \cdot X^0_t(x) + \oO_\infty, \oO_\infty).
\end{align*} 
Analogously we deduce 
\begin{align*}
\Wp(Y^\e_t(x), \mu^\e_*) 
&\gqq \e^{1\wedge p}\cdot \Wp(\e^{-1}\cdot X^0_t(x) + \oO_\infty, \oO_\infty) -\e^{1\wedge p}\cdot \Wp(\yY^x_t, \oO_\infty). 
\end{align*} 
Consequently, 
\begin{equation}\label{eq:NLprofileApprox}
\Big|\frac{\Wp(Y^\e_t(x), \mu^\e_*)}{\e^{1\wedge p}} -  \Wp(\e^{-1} \cdot X^0_t(x) + \oO_\infty, \oO_\infty) \Big| \lqq \Wp(\yY^x_t, \oO_\infty). 
\end{equation}
The right-hand side of \eqref{eq:NLprofileApprox} does not depend of $\e$ and by Lemma~\ref{lem:inhomOU} it tends to $0$ as $t\ra \infty$.
It is therefore enough to study the precise longterm 
behavior of $\Wp(\e^{-1}\cdot X^0_t(x) + \oO_\infty, \oO_\infty)$ in order to derive the cutoff phenomenon.
\subsection{\textbf{Proof of Theorem~\ref{th:linearwhite}}}
\noindent 
For any $0< p< p_*$, $\ft^x_\e$ and $w_\e$ being given in statement and $r\in \RR$, 
\eqref{eq:pivot}, \eqref{prop:aproxlineal}, \eqref{lem:limites}, \eqref{eq:NLprofileApprox} yield
\begin{equation*}\label{eq:splitting}
\begin{split}
&\limsup\limits_{\e\ra 0}\frac{\Wp(X^\e_{\ft^x_\e + r \cdot w_\e} (x),\mu^\e)}{\e^{1\wedge p}}=\limsup\limits_{\e\ra 0}\Wp\Big(\frac{X^0_t(x)}{\e} + \oO_\infty, \oO_\infty\Big),\\
&\liminf\limits_{\e\ra 0}\frac{\Wp(X^\e_{\ft^x_\e + r \cdot w_\e} (x),\mu^\e)}{\e^{1\wedge p}}=\liminf\limits_{\e\ra 0}\Wp\Big(\frac{X^0_t(x)}{\e} + \oO_\infty, \oO_\infty\Big).
\end{split}
\end{equation*}
For short, we define 
\begin{equation}\label{def:lambda}
\fT^x_\e=\ft^x_\e+r\cdot w_\e-\tau^x
\quad \textrm{ and } \quad \Lambda^x(\e):=\frac{(\fT^x_\e)^{\ell-1}}{\e e^{\fq^x \fT^x_\e}} \sum_{k=1}^{m} e^{i  \fT^x_\e\theta^x_k} v^x_k.
\end{equation}
\noindent
\textbf{Claim A.}
\begin{align*}
\limsup_{\e\ra 0}\frac{\Wp(X_{\ft^x_\e+r\cdot w_\e}^\e(x), \mu^\e)}{\e^{1\wedge p}}=\limsup_{\e\ra 0}\Wp\big(
\Lambda^x(\e)
 +  \oO_\infty,  \oO_\infty\big)
\end{align*}
and
\begin{align*}
\liminf_{\e\ra 0}\frac{\Wp(X_{\ft^x_\e+r\cdot w_\e}^\e(x), \mu^\e)}{\e^{1\wedge p}}=\liminf_{\e\ra 0}\Wp\big(
\Lambda^x(\e)
 +  \oO_\infty,  \oO_\infty\big).
\end{align*}
for  any $0<p< p_*$.
In particular, the limit
\begin{align}\label{eq:distancia}
\lim\limits_{\e\ra 0}\frac{\Wp(X_{\ft^x_\e+r\cdot w_\e}^\e(x), \mu^\e)}{\e^{1\wedge p}}\quad \textrm{ exists iff}\quad
\lim_{\e\ra 0}\Wp\big(
\Lambda^x(\e)
 +  \oO_\infty,  \oO_\infty\big)\quad \textrm{exists}.
\end{align}

\noindent
\textit{Proof of Claim A.}
In the sequel we study the asymptotics of the drift term $X^0_t(x) \cdot \e^{-1}$.
A straightforward calculation shows
\begin{equation}\label{eq:newscaling}
\lim\limits_{\e \ra 0} 
\frac{(\fT^x_\e)^{\ell-1} e^{-\fq^x \fT^x_\e}}{\e}=e^{-\fq^x \tau}(\fq^x)^{1-\ell}e^{-\fq^x r\cdot w}.
\end{equation}
The preceding limit implies 
with the help of the spectral decomposition \eqref{eq:hartmangrobman} given in Lemma~\ref{lem:jara29} and  
the triangle inequality that 
\begin{equation}\label{eq:replacement}
\begin{split}
\Wp\Big(\frac{X^0_{\ft^x_\e+r\cdot w_\e}(x)}{\e} +  \oO_\infty,  \oO_\infty\Big)
&\lqq 
 \Wp\Big(\Big(\frac{X^0_{\tau+\fT^x_\e}(x)}{\e}-
\Lambda^x(\e)\Big)+\oO_\infty,  \oO_\infty\Big)+
\Wp\Big(
\Lambda^x(\e)
 +  \oO_\infty,  \oO_\infty\Big).
\end{split}
\end{equation}
We set
\[
R^x_\e:= \Wp\Big(\Big(\frac{X^0_{\tau+\fT^x_\e}(x)}{\e}-
\Lambda^x(\e)\Big)
 +  \oO_\infty,  \oO_\infty\Big).
\]
Analogous reasoning yields
\begin{align*}
\Wp \Big(
\Lambda^x(\e)
 +  \oO_\infty,  \oO_\infty\Big)
 \lqq \Wp\Big(\frac{X^0_{\ft^x_\e+r\cdot w_\e}(x)}{\e} +  \oO_\infty,  \oO_\infty\Big)+R^x_\e.
\end{align*}
In the sequel it remains to show that $R^x_\e\ra 0$ as $\e \ra 0$. By the continuity  of $z\ra \Wp(z+\oO_\infty,\oO_\infty)$ at $z=0$
it is enough to prove 
\[
\Big|\frac{X^0_{\tau+\fT^x_\e}(x)}{\e}-
\Lambda^x(\e)\Big|\ra  0, \quad \e\ra 0,
\]
which is valid due to the 
limit \eqref{eq:hartmangrobman} and  \eqref{eq:newscaling}.
This finishes the proof of Claim A. $\square$

\bigskip
In the sequel, we prove the window cutoff asymptotics in \eqref{eq:infimosupremo}.
Note that $\Lambda^x(\e)$ is uniformly bounded on $\e\in (0,1]$.
For any accumulation point $U$ (as $\e\ra 0$)  of $\big(\Wp (
\Lambda^x(\e)
 +  \oO_\infty,  \oO_\infty)\big)_{\e\in (0,1]}$ there exists a sequence $(\e_k)_{k\in \NN}$, $\e_k\ra 0$ as $k\ra \infty$, such that 
\[
U=\lim\limits_{k\ra \infty}
\Wp \big(
\Lambda^x(\e_k)
 +  \oO_\infty,  \oO_\infty\big).
\]
The Bolzano-Weierstrass theorem for the sequence $(\Lambda(\e_k))_{k\in \NN}$,
the limit \eqref{eq:newscaling}
 and the  continuity of $\Wp$ yield
\begin{align}\label{eq:aculimit}
U=\Wp(e^{-\fq^x \tau^x}(\fq^x)^{1-\ell^x}e^{-\fq^x w r} u+\oO_\infty,\oO_\infty)\quad \textrm{ for some } u\in \omega(x).
\end{align}
In particular,
\begin{equation*}
\begin{split}
&\limsup\limits_{\e \ra 0}\Wp \big(
\Lambda^x(\e)
 +  \oO_\infty,  \oO_\infty\big)=\Wp(e^{-\fq^x \tau^x}(\fq^x)^{1-\ell^x}e^{-\fq^x w r}\hat u+\oO_\infty,\oO_\infty),\\
&\liminf\limits_{\e \ra 0}\Wp \big(
\Lambda^x(\e)
 +  \oO_\infty,  \oO_\infty\big)=\Wp(e^{-\fq^x \tau^x}(\fq^x)^{1-\ell^x}e^{-\fq^x w r}\check u+\oO_\infty,\oO_\infty),
\end{split}
\end{equation*} 
where $\hat u,\check u\in \omega(x)$ and $\check u\neq 0$ by \eqref{eq:belowabovebound}.
Hence
item d) in Lemma~\ref{lem:properties} implies
\begin{align*}
&\lim\limits_{r\ra \infty}\limsup\limits_{\e \ra 0}\Wp \big(
\Lambda^x(\e)
 +  \oO_\infty,  \oO_\infty\big)=0 \quad \textrm{ and }\quad
\lim\limits_{r\ra -\infty}\liminf\limits_{\e \ra 0}\Wp \big(
\Lambda^x(\e)
 +  \oO_\infty,  \oO_\infty\big)=\infty.
\end{align*}
This finishes the proof of Theorem~\ref{th:linearwhite}.

\subsection{\textbf{Proof of Theorem~\ref{th:profileabstract9}}}
We keep the notation \eqref{def:lambda}
of the proof of Theorem~\ref{th:linearwhite}.
By \eqref{eq:distancia} it is enough to prove that the limit 
\begin{equation}\label{eq:limiteexists}
\lim\limits_{\e\ra 0}\Wp \Big(
\Lambda^x(\e)
 +  \oO_\infty,  \oO_\infty\Big)\quad \textrm{ exists}.
\end{equation}
We recall the definition of $\Lambda^x(\e)$ \eqref{def:lambda} and the limit \eqref{eq:newscaling}. By \eqref{eq:aculimit} we have 
\begin{align}
&\left\{
\textrm{accumulation points of }
\Wp \big(
\Lambda^x(\e)
 +  \oO_\infty,  \oO_\infty\big) \textrm{ as } \e \ra 0\right\}
 \nonumber
 \\
&\hspace{2cm}=
\left\{
\Wp\big((e^{-\fq^x \tau^x}(\fq^x)^{1-\ell^x}e^{-\fq^x w r})\, u+\oO_\infty,\oO_\infty\big)
: u\in \omega(x) \right\}.
\label{eq:igualconj}
\end{align}
For $p\gqq 1$,
the shift linearity given in item d) of Lemma~\ref{lem:properties} implies
\begin{equation}\label{eq:shiftlinealp}
\Wp(e^{-\fq^x \tau^x}(\fq^x)^{1-\ell^x}e^{-\fq^x w r} u+\oO_\infty,\oO_\infty)=
e^{-\fq^x \tau^x}(\fq^x)^{1-\ell^x}e^{-\fq^x w r}|u|.
\end{equation}
Combining \eqref{eq:igualconj} and \eqref{eq:shiftlinealp} we infer
\begin{align}
&\left\{
\Wp\big((e^{-\fq^x \tau^x}(\fq^x)^{1-\ell^x}e^{-\fq^x w r})\, u+\oO_\infty,\oO_\infty\big)
: u\in \omega(x) \right\}\\
&\hspace{4cm}=
\left\{
e^{-\fq^x \tau^x}(\fq^x)^{1-\ell^x}e^{-\fq^x w r}\,|u|
: u\in \omega(x) \right\}.\label{eq:igualconj1}
\end{align}
Hence \eqref{eq:igualconj} and 
\eqref{eq:igualconj1}
imply that the limit \eqref{eq:limiteexists} exists if and only if the right-hand side of \eqref{eq:igualconj1} has exactly one element. This is equivalent to $\omega(x)$ being contained in a sphere in $\RR^d$ with respect to the Euclidean distance.
For $p\in (0,1)$ the shift linearity is not valid and we are stuck after  \eqref{eq:igualconj}. Consequently, \eqref{eq:igualconj}
holds true and the limit 
\eqref{eq:limiteexists} exists if and only if for all $\lambda>0$
the  function
\[
\omega(x)\ni u\mapsto \Wp(\lambda u+\oO_\infty,\oO_\infty)\quad \textrm{ is constant}. 
\] 
This finishes the proof of Theorem~\ref{th:profileabstract9}. 

\appendix
\section{\textbf{Existence of the invariant measure}}\label{AP1:existencia}
\subsection{\bf{Invariant distribution $\mu^\e$}}
In the sequel we show the existence of a unique invariant distribution $\mu^\e$ of the solution of \eqref{eq:SDE99} for any $\e>0$.
We stress that beyond 
the existence of moments (Hypothesis~\ref{hyp:moments9}), this does not include any regularity such as absolute continuity whatsoever in our setting. For instance, our setting covers nonlinear oscillators with degenerate noise in Example~\ref{ex:oscillator}.

We recall the standing assumptions Hypothesis~\ref{hyp:dissipa9} with $\delta>0$ and Hypothesis~\ref{hyp:moments9} with $p_*>0$.
For the existence of the invariant probability measure $\mu^\e$ it is enough to verify
the following condition by \cite{PraGaZa}, p. 388.
For some $x\in \RR^d$, the limit
\begin{equation}\label{eq:daprato}
\lim_{R\ra \infty}\liminf_{T\ra \infty}\frac{1}{T}\int_{0}^{T}
\PP\left(|X^\e_t(x)|>R\right)\ud t=0.
\end{equation}
Hypotheses \ref{hyp:dissipa9}  and \ref{hyp:moments9} imply inequality
(D.3) p. 71 in \cite{BHPTV}. That is to say, for
$\gamma \in (0,1\wedge p_*)$
 there exist positive constants $C_1,C_2,C_3$ such that
for all $x\in \RR^d$, $\e>0$, $t\gqq 0$, $A=\e \Pi$, $c=\e$ 
\begin{equation}\label{eq:momentos}
\EE[|X^\e_t(x)|^\gamma]\lqq e^{-\delta \gamma t}|x|^\gamma+C_3,
\end{equation}
where $C_3= c^\gamma+\frac{1}{\gamma \delta}
\big(\gamma \delta c^\gamma+ C_1\|A\|^\gamma+ C_2 c^{\gamma-2}\|A\|^{2}  \big)=\e^\gamma \cdot
\big(2+
\frac{1}{\gamma \delta}(
C_1\|\Pi\|^\gamma+C_2 \|\Pi\|^2)
\big)
$.
Inequality
\eqref{eq:momentos} implies \eqref{eq:daprato} with the help of the Markov inequality. 

For the uniqueness, it enough 
to verify the following condition given in Theorem~11.4.3 in \cite{Kallianpur}.
For any given positive numbers $\eta$, $\delta$ and $R$, there exists a positive constant 
$S$ such that
\begin{equation}\label{eq:unique}
\frac{1}{T}\int_{0}^{T}
\mathbb{P}\left(|X^\e_t(x)-X^\e_t(y)|\gqq \delta\right)\ud t <  \eta \quad \textrm{ for all }\quad |x|,|y|\lqq R \quad \textrm{ and } \quad T>S.
\end{equation}
Hypotheses \ref{hyp:dissipa9}, \ref{hyp:moments9}  and the additivity of the noise imply
(D.5) p. 71 in \cite{BHPTV}. In other words, for any $\gamma\in (0,1\wedge p_*)$, $x,y\in \RR^d$, $t\gqq 0$, $\e>0$, $c=\e$ we have 
\begin{equation*}
\mathbb{E}[|X^\e_t(x)-X^\e_t(y)|^\gamma]\lqq |x-y|^\gamma e^{-\delta \gamma t}+2\e^\gamma.
\end{equation*}
The preceding inequality implies \eqref{eq:unique} with the help of the Markov inequality.

\subsection{\bf{Convergence to $\mu^\e$ in $\W$ for $p_*>0$}}\label{ss:W-ergodicity}
Due to Hypothesis~\ref{hyp:dissipa9} and the additive of the noise the natural coupling yields
\begin{align}\label{eq:x-y}
|X^{\e}_t(x)-X^{\e}_t(y)|\lqq |x-y| e^{-\delta t}\quad \textrm{ for all }\quad x,y\in \RR^d, t\gqq 0.
\end{align}
Since $\mu^\e$ is an invariant measure and $X^\e$ is a Feller process, disintegration and \eqref{eq:x-y} imply
\begin{equation}\label{eq:converlimiteW}
\begin{split}
\W(X^\e_t(x),\mu^{\e})&\lqq 
\int_{\RR^d }\W(X^\e_t(x),X^\e_t(y))\mu^\e(\ud y)\lqq  e^{-({1\wedge p_*})\delta t}
\int_{\RR^d }|x-y|^{1\wedge p_*}\mu^\e(\ud y)\\
&\lqq e^{-({1\wedge p_*})\delta t}|x|^{1\wedge p_*}+e^{-({1\wedge p_*})\delta t}\int_{\RR^d}|y|^{1\wedge p_*}\mu^{\e}(\ud y).
\end{split}
\end{equation}
The preceding right-hand side tends to zero as  $t\ra \infty$ provided that $\int_{\RR^d}|y|^{1\wedge p_*}\mu^{\e}(\ud y)<\infty$ which is shown in (2.84) p. 48 in \cite{BHPTV}.
\section{\textbf{$L^{p}$ estimates for $p\in (0,p_*)$}}
We recall the L\'evy-Khinchin formula of $L$ with characteristic triple $(a,\Sigma,\nu)$ 
\[
\ln(\EE\big[e^{i\< u, L_t\>}\big]) = t\left(i\< a, u\> - \frac{1}{2} \< u, \Sigma u\> + 
\int_{\RR^d} \Big(e^{i\< u, z\>} - 1 - i\< u, z\> \ind_{(0,1)}(|z|)\Big) \nu(\ud z)\right)
\]
and the pathwise L\'evy-It\^o representation  
\begin{equation}\label{eq:LevyIto}
L_t=at+\Sigma^{1/2}B_t+\int_{0}^{t}\int_{|z|\lqq 1} z \ti{N}(\ud s \ud z)+
\int_{0}^{t}\int_{|z|>1} z {N}(\ud s \ud z),
\end{equation}
where $(B_t)_{t\gqq 0}$ is a standard Brownian motion in $\RR^d$, $N$ is a Poisson random measure on $[0,\infty)\times \RR^d$ with intensity measure $\ud t\otimes \nu(\ud z)$ 
and
$\ti N$ is the compensated counterpart of $N$. See \cite{Sa} for further details on L\'evy processes. 

We recall the standing assumptions Hypothesis~\ref{hyp:dissipa9} with $\delta>0$ and Hypothesis~\ref{hyp:moments9} with $p_*>0$.

\subsection{\textbf{Localization}}
We start with the probability estimate of the event 
\begin{equation*}
\dD_t^x=\Big\{\sup\limits_{0\lqq s\lqq t}|\yY_s^x|> \vartheta\Big\}, \qquad \vartheta >0.
\end{equation*}
where $\yY^x$ is given in \eqref{eq: Yxt9}. Note that $\zZ_\cdot(0) = \yY^0_\cdot$ satisfies 
\begin{equation}\label{eq:OUx}
\ud \zZ_t(x) = - Db(0) \zZ_t(x) \ud t + \ud  L_t, \qquad \zZ_0(x) = x
\end{equation}
for $x= 0$.   
\begin{lem}\label{lem: Siorpaes}
For any $\gamma\in (0,p_*\wedge 1]$ there is a positive constant $C$ such that for any $\vartheta \gqq 1$, $x\in \RR^d$ and $t\gqq 0$  we have 
\begin{equation}\label{eq:siorpaesnew}
\PP(\dD_t^x)\lqq C\,t \vartheta^{-\gamma }.
\end{equation}
\end{lem}
\begin{proof}
By Theorem~1 in \cite{SIORPAES} we have
\begin{align*}
\sup\limits_{0\lqq s\lqq t}|\yY^x_s|\lqq  6\sqrt{[\yY^x_\cdot(0)]_s}+2\int_{0}^{t} H_{s-}\, \ud L_s,\qquad 
\mbox{ where }\qquad 
H_{s-}=\frac{\yY^x_{s-}}{\sqrt{\sup\limits_{s\lqq t} (|\yY^x_{s-}|^2+[\yY^x_\cdot]_{s-})}}.
\end{align*}
In particular, it follows
\begin{align*}
&[\yY^x_\cdot]_{t} = [L]_{t} = \int_0^t \int_{|z|\lqq 1} |z|^2 N(\ud s\ud z)\quad \textrm{such that }\\
&\int_0^{t}  H_{s-}\, \ud \yY^x_s = \int_0^t \< H_{s}, - Db(0) \yY^x_s \> \ud s + 
\int_0^t \int_{|z|\lqq 1}\< H_{s-}, z\> \ti N(\ud s\ud z) +\int_0^t \int_{|z|> 1}\< H_{s-}, z\> N(\ud s\ud z).
\end{align*}
By Hypothesis~\ref{hyp:dissipa9} we obtain 
$
\int_0^t \< H_{s-}, - Db(X^0_s(x)) \yY^x_s \> \ud s \lqq 0 
$ a.s. 
Hence 
\begin{align*}
&\PP\Big( \sup_{0\lqq s\lqq t}|\yY^x_s| > \vartheta\Big)\\ 
&\lqq \PP\Big(6 \Big(\int_0^{t} \int\limits_{|z|\lqq 1} |z|^2 N(\ud s\ud z)\Big)^{\nicefrac{1}{2}}  +  2\int_0^{t} \int\limits_{|z|\lqq 1}\< H_{s-}, z\> \ti N(\ud s\ud z) + 
2\int_0^{t} \int\limits_{|z|> 1}\< H_{s-}, z\> N(\ud s\ud z)  > \vartheta\Big) \\
&\lqq \PP\Big(\int_0^{t} \int_{|z|\lqq 1} |z|^2 N(\ud s\ud z) > \frac{\vartheta^2}{18^2} \Big)  + \PP\Big(\int_0^{t} \int_{|z|\lqq 1}\< H_{s-}, z\> \ti N(\ud s\ud z) > \frac{2\vartheta}{3} \Big) \\
&\qquad + \PP\Big(\int_0^{t} \int_{|z|> 1}\< H_{s-}, z\> N(\ud s\ud z) > \frac{2\vartheta}{3} \Big).
\end{align*}
We continue term by term. By the Chebyshev inequality we obtain  
\begin{align*}
\PP\Big(\int_0^{t} \int_{|z|\lqq 1} |z|^2 N(\ud s\ud z) > \frac{\vartheta^2}{18^2} \Big) 
&\lqq  \frac{18^2 t }{\vartheta^2} \int_{|z|\lqq 1} |z|^2 \nu(\ud z) =:C_1  \frac{t}{\vartheta^2}
\end{align*}
and 
\begin{align*}
\PP\Big(\int_0^{t} \int_{|z|\lqq 1}\< H_{s-}, z\> \ti N(\ud s\ud z) > \frac{2\vartheta}{3}\Big) 
&\lqq \Big(\frac{3}{2}\Big)^2\frac{1}{\vartheta^2} \EE\Big[\Big(\int_0^{t} \int_{|z|\lqq 1}\< H_{s-}, z\> \ti N(\ud s\ud z\Big)^2\Big]\nonumber\\
&= \Big(\frac{3}{2}\Big)^2\frac{1}{\vartheta^2} \EE\Big[\int_0^{t} \int_{|z|\lqq 1}\< H_{s-}, z\>^2 \nu(\ud z)\ud s\Big]
\nonumber\\
&\lqq  \Big(\frac{3}{2}\Big)^2\frac{t}{\vartheta^2}   \int_{|z|\lqq 1} |z|^2 \nu(\ud z)
=:C_2  \frac{t}{\vartheta^2}.
\end{align*}
Finally, for $\gamma\in (0,p_*\wedge 1]$ we have 
\begin{align*}
\PP\Big(\int_0^{t} \int_{|z|> 1}\< H_{s-}, z\> N(\ud s\ud z) > \frac{2\vartheta}{3}\Big) 
&\lqq \PP\Big(\int_0^{t} \int_{|z|> 1} |z| N(\ud s\ud z) > \frac{2\vartheta}{3}\Big) \nonumber\\
&\lqq \Big(\frac{3}{2}\Big)^\gamma \frac{1}{\vartheta^\gamma}
\EE\Big[\Big(\int_0^{t} \int_{|z|> 1} |z| N(\ud s\ud z)\Big)^{\gamma}\Big] \nonumber\\
&\lqq \Big(\frac{3}{2}\Big)^\gamma \frac{1}{\vartheta^\gamma}
\EE\Big[\int_0^{t} \int_{|z|> 1} |z|^{\gamma} N(\ud s\ud z)\Big] 
\\
&= \Big(\frac{3}{2}\Big)^\gamma \frac{t}{\vartheta^\gamma} \int_{|z|> 1} |z|^{\gamma} \nu(\ud z)
=:C_3 \frac{t}{\vartheta^\gamma},
\end{align*}
where we have used the subadditivity of the power $\gamma$ in the sense of 
Subsection~1.1.2, see formula (1.6) in \cite{Bie}.  
This finishes the proof of the statement.
\end{proof}
\subsection{\textbf{First order approximation}}\label{app:B2}
We start with some technical preliminaries. 
In order to overcome that $u \mapsto |u|^{p}$ for $p \in (0, 2)$ is not twice continuously differentiable 
which turns out to be necessary for applying It\^o's formula 
we use the following $\cC^2$ norm approximation $|x|_c:= \sqrt{|x|^2 + c^2}, c> 0$, 
with the limiting case $|x|_0 = |x|$. It is well-behaved in the following 
sense. For any $c>0$ we have 
\[
c \lqq |x|_c \lqq |x| + c, \quad \nabla |x|_c := \frac{x}{|x|_c} 
\quad \textrm{ and }\quad 0 \lqq \frac{|x|}{|x|_c} < 1.
\]
Furthermore, it is straightforward to verify for 
$G(x)=|x|^p_c$ the following calculations
\[
\nabla G(x)=p |x|^{p-1}_c \frac{x}{|x|_c}=p|x|^{p-2}_c x\quad \textrm{ and } \quad|\nabla G(x)|\lqq p|x|^{p-1}_c.
\]
The $L_1$-matrix norm $\|\cdot \|_1$ of the respective Hessian $H_{G}(x)$, $x\in \RR^d$, can be estimated as follows   
\begin{align*}
\|H_{G}(x)\|_1
\lqq p d|x|^{p-2}_c+pd (2-p)|x|^{p-2}_c=C(p,d)|x|^{p-2}_c.
\end{align*}
For details of the estimates, we refer to p. 69 in \cite{BHPTV}.
Since $p \in (0,2)$ and $c\lqq |x|_c$, we obtain
\begin{equation}\label{eq:hessianest}
\sup_{x\in \RR^d}|H_{G}(x)|_1\lqq C(p,d)\,c^{p-2}.
\end{equation}

\medskip

\begin{prop}\label{cl:linearization}
We keep the notation of Theorem~\ref{th:linearwhite}. Then for any $x\in \mathbb{R}^d$, $r\in \RR$ 
and $p\in (0,p_*)$
it follows 
\begin{equation}\label{eq:WpXY}
\lim_{\e\ra 0}\frac{\Wp(X^\e_{\ft^x_\e + r \cdot w_\e}(x),Y^\e_{\ft^x_\e + r \cdot w_\e}(x))}{\e^{1\wedge p}}=0.
\end{equation}
\end{prop}
\begin{proof} By the domination property of the Wasserstein distance in 
Lemma~\ref{lem:properties} it is enough to show the preceding limit in the respective $L^{p}$ space. 
By \eqref{def:Yprocess9}
we have 
\begin{align*}
\ud Y^\e_t(x) &= \big(- Db(X^0_t(x))  Y^\e_t(x) + Db(X^0_t(x)) X^0_t(x)-b(X^0_t(x)) \big)\ud t + \e  \ud L_t.
\end{align*}
Let $\Delta^\e_t := X^\e_t(x)-Y^\e_t(x)$, $t\gqq 0$. Then 
\begin{align*}
\ud\, \Delta^\e_t  
&= -\big(b(X^\e_t(x))-b(Y^\e_t(x))\big)\ud t -\big(b(Y^\e_t(x)) -b(X^0_t(x)) - Db(X^0_t(x))  \e\yY^x_t \big)\ud t,
\end{align*}
where $(\yY^x_t)_{t\gqq 0}$ is given in \eqref{eq: Yxt9}.
An elementary estimate of the $p_*$-th power of a sum yields for all $t\gqq 0$
\begin{align}
\EE[|\Delta^\e_t|^{p_*}]
&=\EE[|(X^\e_t(x)-X^0_t(x))+\e \yY^x_t|^{p_*}]\lqq 
C_{p_*}\left(\EE[|X^\e_t(x)-X^0_t(x)|^{p_*}]+\e^{p_*}\EE[|\yY^x_t|^{p_*}]\right),
\label{eq:potencias}
\end{align}
where $C_{p_*}$ is a positive constant.
Since $(\yY^x_t)_{t\gqq 0}$ satisfies a dissipative 
linear equation, it exhibits the same integrability as $L$, 
which is straightforward to verify. There are a positive constant $\ti C_{p_*}$ and 
a function $S_{p_*}(t)$ of at most polynomial order such that
\begin{equation}\label{eq:ec1}
\EE[|\yY^x_t|^{p_*}] \lqq \ti C_{p_*} \EE[|L_t|^{p_*}]\lqq \ti C_{p_*} S_{p_*}(t) \quad \textrm{ for all } t\gqq 0.
\end{equation}
For the first term of the right-hand side of \eqref{eq:potencias},  
 Lemma~\ref{lem:p>2} and Lemma~\ref{lem:p<2} 
yield the following estimate.
For any $\eta\in (0,p_*)$ there is a map
$R_\eta:[0,\infty)\ra [0,\infty)$ 
which increases with polynomial order as $t$ tends to infinity,
  such that
\begin{align}\label{eq:ec2}
\EE[|X^\e_t(x)-X^0_t(x)|^{p_*}]\lqq \e^{p_*-\eta}R_\eta(t)
\quad \textrm{ for any } t\gqq 0.
\end{align}
\noindent
\textbf{We start with the case $p_*>1$ and $p\in (1,p_*)$. }
The H\"older inequality implies
\begin{align}\label{eq:p1mo}
\EE[|\Delta^\e_t|^{p-1}]&\lqq
\left(\EE[|\Delta^\e_t|^{p_*}]\right)^{\frac{p-1}{p_*}}\lqq \e^{\frac{p_*-\eta}{p_*}(p-1)}\ti R_\eta(t)=
\e^{p-1-\eta'}\ti R_\eta(t)\quad \textrm{ for any  } t\gqq 0,
\end{align}
where $\ti R_\eta$ is a function of at most polynomial order as $t$ tends to infinity and 
$\eta'=\eta (p-1) p_*^{-1}$.
For $\eta$ small enough we fix $\eta'\in (0,\nicefrac{1}{4})$.
Since $p_*>1$, we may choose $p\in (1,p_*)$ and $\theta\in (0,\nicefrac{1}{4})$. We split 
\begin{align}\label{ec:pivote33}
\EE[|\Delta^\e_t|^{p}] 
&= \EE[|\Delta^\e_t|^{p}\;\ind(\mathcal{A}^\e_t)]
+\EE[|\Delta^\e_t|^{p}\;\ind((\mathcal{A}^\e_t)^{\cc})],
\end{align}
where 
\begin{equation}\label{eq:evento}
\mathcal{A}^\e_t:=\Big\{\sup_{0\lqq s\lqq t} |\e \yY^x_s|\lqq \e^{1-\theta}\Big\}.
\end{equation}
First we prove that 
\[
\left(\EE[|\Delta^\e_t|^{p}\ind(\mathcal{A}^\e_t)]\right)^{1/p}\lqq \left(pC(|x|)  \int_{0}^{t}\ti R_\eta(s)\ud s\right)^{1/p} \e^{1+\frac{1-\eta'-2\theta}{p}},
\]
where $C(|x|)=\max\limits_{|u|\lqq |x|+1}|D^2b(u)|$. The choice of $\eta'$ and $\theta$ yields $1-\eta'-2\theta>\nicefrac{1}{4}$.
For notational convenience, we use the differential 
formalism, however, we stress that all differential inequalities are understood in the integral sense.
Since $p> 1$, the chain rule, Hypothesis~\ref{hyp:dissipa9} and
 Cauchy-Schwarz inequality
imply
\begin{align*}
\ud\,|\Delta^\e_t|^{p}
&=-p|\Delta^\e_t|^{p-2}\<\Delta^\e_t,b(X^\e_t(x))-b(Y^\e_t(x))\>\ud t\\
&\qquad -p|\Delta^\e_t|^{p-2}\<\Delta^\e_t,b(Y^\e_t(x)) -b(X^0_t(x)) - Db(X^0_t(x))\e \yY^x_t)\>\ud t\\
&\lqq 
-\delta p|\Delta^\e_t|^{p}+p|\Delta^\e_t|^{p-1}
|b(Y^\e_t(x)) -b(X^0_t(x)) - Db(X^0_t(x))\e \yY^x_t)|\ud t.
\end{align*}
On the event $\mathcal{A}^\e_t$, Taylor's theorem applied to $b$ implies
\begin{align*}
\ud\, |\Delta^\e_t|^{p}\lqq 
-\delta p |\Delta^\e_t|^{p}\ud t + pC(|x|)|\Delta^\e_t|^{p-1}\e^{2-2\theta}\ud t.
\end{align*}
Taking expectation, the integral monotonicity, Fubini's theorem and \eqref{eq:p1mo} yield
\begin{align*}
\ud\,\EE[|\Delta^\e_t|^{p}\ind(\mathcal{A}^\e_t)]&\lqq 
-\delta p \EE[|\Delta^\e_t|^{p}\ind(\mathcal{A}^\e_t)]\ud t+
 pC(|x|)\EE[|\Delta^\e_t|^{p-1}\ind(\mathcal{A}^\e_t)]\e^{2-2\theta}\ud t\\
&\lqq 
pC(|x|)\EE[|\Delta^\e_t|^{p-1}]\e^{2-2\theta}\ud t
\\
&\lqq
pC(|x|) \ti R_{\eta}(t) \e^{p+1-\eta'-2\theta}\ud t.
\end{align*}
Bearing in mind  $|\Delta^\e_0|^{p}=0$, we have
\[
\EE[|\Delta^\e_t|^{p}\ind(\mathcal{A}^\e_t)]\lqq pC(|x|)  \e^{p+1-\eta'-2\theta}\int_{0}^{t}\ti R_\eta(s)\ud s.
\]
Therefore
\begin{equation}\label{ec:est1}
\left(\EE[|\Delta^\e_t|^{p}\ind(\mathcal{A}^\e_t)]\right)^{1/p}\lqq \left(p C(|x|)  \int_{0}^{t}\ti R_\eta(s)\ud s\right)^{1/p} \e^{1+\frac{1-\eta'-2\theta}{p}}.
\end{equation}
We continue with the estimate on the complement of $\mathcal{A}^\e_t$. We show
\[
\EE[|\Delta^\e_t|^{p}\ind((\mathcal{A}^\e_t)^{\cc})]\lqq
\e^{p-\eta'}\rR(t)\cdot \PP\big((\mathcal{A}^\e_t)^{\cc}\big)^\frac{p_*-p}{p_*},
\]
where $\rR(t)$ is a function of at most polynomial order. Indeed,
by  H\"older's inequality and the inequalities \eqref{eq:potencias}, \eqref{eq:ec1} and \eqref{eq:ec2}  we have
\begin{align*}
\EE[|\Delta^\e_t|^{p}\ind((\mathcal{A}^\e_t)^{\cc})]&
\lqq \EE[|\Delta^\e_t|^{p_*}]^{\frac{p}{p_*}}\cdot
\PP\big((\mathcal{A}^\e_t)^{\cc}\big)^\frac{p_*-p}{p_*}\\
&\lqq \Big(
C_{p_*}\e^{p_*-\eta}R_\eta(t)+C_{p_*}\e^{p_*}\ti C_{p_*}S_{p_*}(t)
\Big)^{\frac{p}{p_*}}
\cdot \PP\big((\mathcal{A}^\e_t)^{\cc}\big)^\frac{p_*-p}{p_*}\\
&\lqq \Big(\left(
C_{p_*}\e^{p_*-\eta}R_\eta(t)\right)^{\frac{p}{p_*}}
+\left(C_{p_*}\e^{p_*}\ti C_{p_*}S_{p_*}(t)
\right)^{\frac{p}{p_*}}\Big)
\cdot \PP\big((\mathcal{A}^\e_t)^{\cc}\big)^\frac{p_*-p}{p_*}\\
&= \Big(\left(
C_{p_*}R_\eta(t)\right)^{\frac{p}{p_*}}\e^{(p_*-\eta) \frac{p}{p_*}}
+\left(C_{p_*}\ti C_{p_*}S_{p_*}(t)
\right)^{\frac{p}{p_*}}\e^{p}\Big)
\cdot \PP\big((\mathcal{A}^\e_t)^{\cc}\big)^\frac{p_*-p}{p_*}\\
&\lqq \e^{p-\eta \frac{p}{p_*}}\rR(t)\cdot \PP\big((\mathcal{A}^\e_t)^{\cc}\big)^\frac{p_*-p}{p_*},
\end{align*}
where $ \mathcal{R}(t):=\max\{\big(
C_{p_*}R_\eta(t)\big)^{\frac{p}{p_*}},\big(C_{p_*}\ti C_{p_*}S_{p_*}(t)
\big)^{\frac{p}{p_*}}\}$.
As a consequence,
\begin{align}\label{ec:est2}
\left(\EE[|\Delta^\e_t|^{p}\ind((\mathcal{A}^\e_t)^{\cc})]
\right)^{1/p}
\lqq \e^{1-\frac{\eta}{p_*}} 
(\rR(t))^{\frac{1}{p}}
\cdot
\PP\left((\mathcal{A}^\e_t)^{\cc}
\right)^\frac{p_*-p}{p_*p}.
\end{align}
Combining estimates \eqref{ec:est1}, \eqref{ec:est2} in decomposition \eqref{ec:pivote33} 
we obtain a positive constant $C:=C(p_*,p,\delta, |x|, |D^2F|)$ such that
for any $t\gqq 0$
\begin{align}\label{ec:desigualdad}
\quad \Wp(X^{\e}_t(x),Y^{\e}_t(x))
&\lqq 
\left(\EE[|\Delta^\e_t|^{p}]\right)^{\nicefrac{1}{p}}\\
&\lqq 
\left(pC(|x|)  \int_{0}^{t}\ti R_\eta(s)\ud s\right)^{1/p} \e^{1+\frac{1-\eta'-2\theta}{p}}+\e^{1-\frac{\eta}{p_*}} 
(\rR(t))^{\frac{1}{p}}
\cdot
\PP\left((\mathcal{A}^\e_t)^{\cc}\right)^\frac{p_*-p}{p_*p}.\nonumber
\end{align}
By Lemma~\ref{lem: Siorpaes} there exists 
a positive constant $C$ such that for all $\gamma\in (0, 1)$ for the choice $\vartheta = \e^{-\theta/\gamma}$ and any $t\gqq 0$ 
it follows 
\begin{equation}\label{eq:siorpaesnew9}
\PP((\mathcal{A}^\e_t)^{\cc})\lqq Ct \e^{\theta}.
\end{equation}
We further restrict $\theta$ such that additionally
$0<\theta<\min\{\frac{2\eta p}{p_*-p},\nicefrac{1}{4}\}$.
Hence, with the help of inequality \eqref{ec:desigualdad} and \eqref{eq:siorpaesnew9} we have
\begin{align*}
&\Wp(X^{\e}_t(x),Y^{\e}_t(x))\lqq 
\rR_1(t)\e^{1+\frac{1}{4p}}+\rR_2(t)\e^{1+\frac{\eta}{p_*}},
\end{align*}
where $\rR_1$ and $\rR_2$ are functions of at most polynomial order.
Consequently we obtain the desired limit
\[
\lim_{\e\ra 0}\frac{\Wp(X^\e_{\ft^x_\e + r \cdot w_\e}(x),Y^\e_{\ft^x_\e + r \cdot w_\e}(x))}{\e}=0.
\]
\textbf{We continue with the case $p_*>0$ and $p\in (0,1\wedge p_*]$.} 
Let $\theta \in (0,\nicefrac{1}{4})$ and recall the event
$\mathcal{A}^\e_t$ in \eqref{eq:evento}.
For $p\in (0,1\wedge p_*]$ we split
\begin{align*}
\EE[|\Delta_t^\e|^{p}]=
\EE[|\Delta_t^\e|^{p}\ind(\mathcal{A}^\e_t)]+\EE[|\Delta_t^\e|^{p}\ind((\mathcal{A}^\e_t)^{\cc})]=:J_1+J_2.
\end{align*}
We start with the term $J_1$. Since $|\cdot|^p$ is not differentiable, we apply 
the chain rule for the smooth approximation $|x|^{p}_c=(\sqrt{|x|^2+c^2})^{p}$. 
Hypothesis~\ref{hyp:dissipa9} then yields
\begin{align*}
\ud\,|\Delta^\e_t|^{p}_{c}&=
-p|\Delta^\e_t|^{p-2}_c\<\Delta^\e_t,b(X^\e_t(x))-b(Y^\e_t(x))\>\ud t\\
&\qquad +p|\Delta^\e_t|^{p-2}_c\<\Delta^\e_t,b(Y^\e_t(x)) -b(X^0_t(x)) - Db(X^0_t(x))  \e\yY^x_t
\>\ud t\\
&\lqq -p\delta|\Delta^\e_t|^{p-2}_c|\Delta^\e_t|^2 \ud t
+p|\Delta^\e_t|^{p-1}_c |b(Y^\e_t(x)) -b(X^0_t(x)) - Db(X^0_t(x))  \e\yY^x_t|\ud t\\
&\lqq -p\delta|\Delta^\e_t|^{p}_c \ud t+p\delta c^{p}\ud t
+pc^{p-1} |b(Y^\e_t(x)) -b(X^0_t(x)) - Db(X^0_t(x))  \e\yY^x_t|\ud t.
\end{align*}  
Due to  $|X^0_t(x)|\lqq e^{-\delta t}|x|$ for all $t\gqq 0$ and $x\in \RR^d$,
 Taylor's expansion for $b$ 
on the event $\mathcal{A}^\e_t$ 
 implies
\begin{align*}
\ud\,|\Delta^\e_t|^{p}_{c}\lqq -p\delta|\Delta^\e_t|^{p}_c \ud t+p\delta c^{p}\ud t
+pc^{p-1}C(|x|)\e^{2(1-\theta)},
\end{align*}
where $C(|x|)=\max\limits_{|u|\lqq |x|+1}|D^2b(u)|$.
Hence
\begin{align*}
\ud\, \EE[|\Delta^\e_t|^{p}_{c}\ind(\mathcal{A}^\e_t)]\lqq -p\delta\EE[|\Delta^\e_t|^{p}_c\ind(\mathcal{A}^\e_t)]\ud t+p\delta c^{p}\ud t
+pc^{p-1} C(|x|) \e^{2(1-\theta)}\ud t.
\end{align*}
The integral version of the Gr\"onwall inequality with negative linearity given 
in Lemma~1 in \cite{Mikami} implies for all $t\gqq 0$
\begin{align}\label{eq:cpivotal}
\EE[|\Delta^\e_t|^{p}\ind(\mathcal{A}^\e_t)]\lqq 
\EE[|\Delta^\e_t|^{p}_{c}\ind(\mathcal{A}^\e_t)]\lqq 
c^{p}
+\frac{1}{\delta}c^{p-1} C(|x|)\e^{2(1-\theta)}.
\end{align}
For $p\neq 1$ we have the following.
Since $c>0$ is arbitrary and $\theta\in (0,\nicefrac{1}{4})$,
the choice $c=\e^{1+\nicefrac{\eta}{p}}$ with $\eta\in (0,\frac{p}{2(1-p)})$ in \eqref{eq:cpivotal} 
yields for any $r\in \RR$
\begin{equation}\label{eq:limJ2}
\lim_{\e\ra 0}
\frac{1}{\e^{p}}\EE[|\Delta^\e_{\ft^x_\e + r \cdot w_\e}|^{p}\ind(\mathcal{A}^{\e}_{\ft^x_\e + r \cdot w_\e})]=0.
\end{equation}
The case of $p=1$ follows by the  choice  $c=\e^{2}$ in \eqref{eq:cpivotal}.

We continue with the term $J_2$.
By the subadditivity of the power $p\lqq 1$ and the H\"older inequality 
for  the index $p'/p$ where 
$p'\in (p,p_*)$ and $r$ is such that  $p/p'+1/r=1$ we have
\begin{align}
\EE[|\Delta^\e_t|^{p}\ind((\mathcal{A}^{\e}_t)^{\cc})]&\lqq 
\EE[|X^\e_t(x)|^{p}\ind((\mathcal{A}^{\e}_t)^{\cc})]+\EE[|Y^\e_t(x)|^{p}\ind((\mathcal{A}^{\e}_t)^{\cc})]\nonumber\\
&
\lqq 
(\EE[|X^\e_t(x)|^{p'}])^{p/p'}(\PP((\mathcal{A}^{\e}_t)^{\cc}))^{1/r}+(\EE[|Y^\e_t(x)|^{p'}])^{p/p'}(\PP((\mathcal{A}^{\e}_t)^{\cc}))^{1/r}.\label{eq:holder998}
\end{align} 
By Lemma~\ref{lem:p<2}
we obtain for all $t\gqq 0$
\begin{align}\label{eq:ordenXe}
\qquad (\EE[|X^\e_t(x)|^{p'}])^{\nicefrac{p}{p'}}&\lqq 
(
\EE[|X^\e_t(x)-X^0_t(x)|^{p'}]+
|X^0_t(x)|^{p'}))^{p/{p'}}\lqq \e^{p}(1+C_{p'}\cdot t)^{\nicefrac{p}{p'}}+
|X^0_t(x)|^{p}.
\end{align}
Note that for all $t\gqq 0$ it follows
\begin{equation}\label{eq:ordenYe}
(\EE[|Y^\e_t(x)|^{p'}])^{\nicefrac{p}{p'}}\lqq 
 \e^{p}(\EE[|\yY^x_t|^{p'}])^{\nicefrac{p}{p'}}+|X^0_t(x)|^{p}.
\end{equation}
Lemma~A.1 in \cite{BHPTV} yields the existence of a positive constant $C(r,|x|)$ such that 
\begin{equation}\label{eq:ordenepsilon}
|X^0_{\ft^x_\e + r \cdot w_\e}(x)|\lqq C(r,|x|)\e.
\end{equation}
Combining \eqref{eq:holder998} with inequalities
\eqref{eq:siorpaesnew9}, \eqref{eq:ordenXe}, \eqref{eq:ordenYe} and \eqref{eq:ordenepsilon}
gives
\begin{align*}
\EE[|\Delta^\e_{\ft^x_\e + r \cdot w_\e}|^{p}\ind((\mathcal{A}^{\e}_{\ft^x_\e + r \cdot w_\e})^{\cc})]
&\lqq (C(\ft^x_\e + r \cdot w_\e) \e^{\theta})^{1/r}C^{p}(r,|x|)\e^{p}\\
&\quad  +
(C(\ft^x_\e + r \cdot w_\e) \e^{\theta})^{1/r}
\big(\e^{p}(1+C_{p'}\cdot (\ft^x_\e + r \cdot w_\e))^{\nicefrac{p}{p'}}\big)\\
&\quad +(C(\ft^x_\e + r \cdot w_\e) \e^{\theta})^{1/r}\big(
\e^{p}(\EE[|\yY^x_{\ft^x_\e + r \cdot w_\e}|^{p'}])^{\nicefrac{p}{p'}}
\big).
\end{align*} 
Since $\EE[|\yY^x_{t}|^{p'}]\lqq \rR(t)$, where $\rR $ is a function of at most polynomial order, we have
\begin{align*}
&\limsup\limits_{\e \ra 0}
\frac{1}{\e^{p'}}\EE[|\Delta^\e_{\ft^x_\e + r \cdot w_\e}|^{p'}\ind((\mathcal{A}^{\e}_{\ft^x_\e + r \cdot w_\e})^{\cc})]
\\
&\quad \lqq \limsup\limits_{\e \ra 0}
\e^{\theta/r}
(C(\ft^x_\e + r \cdot w_\e))^{1/r}\big(C^p(r,|x|)+(1+C_{p'}\cdot (\ft^x_\e + r \cdot w_\e))^{\nicefrac{p}{p'}}
+\mathcal{R}(\ft^x_\e + r \cdot w_\e)
\big).
\end{align*}
The right-hand side of the preceding inequality equals zero.
The preceding argument combined with \eqref{eq:limJ2} yields the desired limit \eqref{eq:WpXY}.
\end{proof}
\subsection{\textbf{Asymptotic first order approximation}}
\begin{lem}\label{cl:limit}
For any $p\in (0,p_*)$ we have 
\[
\lim_{\e\ra 0}\frac{\Wp(\mu^\e_*,\mu^\e)}{\e^{1\wedge p}}=0.
\]
\end{lem}
\begin{proof}
First we observe that $Y^\e_t(0) = \zZ^\e_t(0)$ for any $t\gqq 0$, $\e>0$, 
where $(\zZ_t^\e(0))_{t\gqq 0}$ is given in \eqref{eq:OUx}.
In abuse of notation, we write $(X^\e_t(\mu^\e))_{t\gqq 0}$ (and analogously respectively $(\zZ^\e_t(\mu^\e_*))_{t\gqq 0}$) for the process starting at the random vector with distribution $\mu^\e$  independent of the noise process $L$.
Since $X^{\e}_t(\mu^\e)=\mu^\e$ and 
$\zZ^{\e}_t(\mu^\e_*)=\mu^\e_*$ for any $t\gqq 0$, the triangle inequality yields
\begin{align}\label{eq:ine0}
\quad \Wp(\mu^\e,\mu^\e_*)&=\Wp(X^\e_t(\mu^\e),\zZ^\e_t(\mu^\e_*))\lqq 
\Wp(X^\e_t(\mu^\e),X^\e_t(0))+
\Wp(X^\e_t(0),\zZ^\e_t(0))+
\Wp(\zZ^\e_t(0),\zZ^\e_t(\mu^\e_*)).
\end{align}
By Proposition~\ref{cl:linearization} for $x= 0$, we have
\begin{equation}\label{eq:term3}
\lim_{\e\ra 0}\frac{\Wp(X^\e_{\ft^x_\e + r \cdot w_\e}(0),\zZ^\e_{\ft^x_\e + r \cdot w_\e}(0))}{\e^{1\wedge p}}=0.
\end{equation}
By disintegration,  inequalities \eqref{eq:x-y} and
(2.84) in \cite{BHPTV}
imply
\begin{align*}
\Wp(X^\e_t(\mu^\e),X^\e_t(0))
&\lqq \int_{\RR^d}
\Wp(X^\e_t(u),X^\e_t(0))\mu^{\e}(\ud u)\lqq e^{-\delta(1\wedge p) t}\int_{\RR^d}
|u|^{1\wedge p}\mu^{\e}(\ud u)\lqq C e^{-\delta(1\wedge p) t}\e^{1\wedge p}
\end{align*} 
for some positive constant $C$.
As a consequence, 
\begin{equation}\label{eq:term1}
\lim\limits_{\e\ra 0}\frac{\Wp(X^\e_{\ft^x_\e+r\cdot w_\e}(\mu^\e),X^\e_{\ft^x_\e+r\cdot w_\e}(0))}{\e^{1\wedge p}}=0.
\end{equation}
Analogously, 
\begin{equation}\label{eq:term2}
\lim\limits_{\e\ra 0}\frac{\Wp(\zZ^\e_{\ft^x_\e+r\cdot w_\e}(\mu^\e_*),\zZ^\e_{\ft^x_\e+r\cdot w_\e}(0))}{\e^{1\wedge p}}=0.
\end{equation}
Combining \eqref{eq:ine0} with the estimates \eqref{eq:term3},  
\eqref{eq:term1} and \eqref{eq:term2} completes the proof.  
\end{proof}
\begin{lem}\label{lem:inhomOU}
For any $p\in (0,p_*)$ we have 
\begin{equation}\label{eq: inhomOUthermState}
\lim\limits_{t\ra \infty}\Wp(\yY^x_t, \oO_\infty)=0.
\end{equation}
\end{lem}
\begin{proof} 
Recall that $\oO_\infty$ is the limiting and invariant distribution of 
the homogeneous Ornstein-Uhlenbeck process $(\zZ(x)_t)_{t\gqq 0}$ defined in \eqref{eq:OUx}. 
That is $\oO_\infty \stackrel{d}{=} \zZ_\infty$. 
Since $-Db(X^0_t(x))$ converges exponentially fast to  $-Db(0)$, it is natural 
to expect that the flow of $(\yY^x_t)_{t\gqq 0}$ behaves as the flow of 
$(\zZ_t(x))_{t\gqq 0}$ for large $t$. 
In \cite{BHPTV}, Lemma~C.3, it is shown that $\yY^x_t \ra  \oO_\infty$ as $t \ra  \infty$ in law. 
However, the law $\oO_\infty$ is not invariant under the random dynamics of $(\yY^x_t)_{t\gqq 0}$ 
due to the time inhomogeneity. 
Analogously as in  
\eqref{eq:converlimiteW} we deduce
\begin{equation}\label{eq: OUthermalizationW2}
\Wp(\zZ_t(x), \oO_\infty) \ra 0, \quad \mbox{ as } t\ra  \infty. 
\end{equation}
We start with the proof of the statement. The triangle inequality yields
\begin{equation}\label{eq: triangulo}
\Wp(\yY^x_t,\oO_\infty)\lqq 
\Wp(\yY^x_t,\zZ_t(0))+
\Wp(\zZ_t(0),\oO_\infty),
\end{equation}
where the second term on the right-hand side 
tends to $0$ as $t\ra  \infty$ due to \eqref{eq: OUthermalizationW2}.
Thus it remains to prove  $
\Wp(\yY^x_t,\zZ_t(0)) \ra  0$,  as  $t \ra  \infty$.  Since 
\begin{equation*}
\Wp(\yY^x_t,\zZ_t(0))\lqq (\EE[|\yY^x_t-\zZ_t(0)|^{p}])^{1\wedge (1/p)},
\end{equation*}
we derive the respective $L^{p}$ estimates.
By \eqref{eq: Yxt9} and \eqref{eq:OUx} we obtain
\[
\ud\,(\yY^x_t-\zZ_t(0))=-Db(X^0_t(x))(\yY^x_t-\zZ_t(0))\ud t+
(Db(0)-Db(X^0_t(x)))\zZ_t(0)\ud t.
\]
We first consider the case $p_*>1$ and $p\in (1,p_*)$. The chain rule and Hypothesis~\ref{hyp:dissipa9} yield
\begin{align*}
\ud  |\yY^x_t-\zZ_t(0)|^{p}&=-p|\yY^x_t-\zZ_t(0)|^{p-2}\<\yY^x_t-\zZ_t(0),Db(X^0_t(x))(\yY^x_t-\zZ_t(0))\>\ud t\\
&\qquad +p|\yY^x_t-\zZ_t(0)|^{p-2}\<\yY^x_t-\zZ_t,(Db(0)-Db(X^0_t(x)))\zZ_t(0)\>\ud t\\
&\lqq -p\delta |\yY^x_t-\zZ_t(0)|^{p}\ud t+ p|\yY^x_t-\zZ_t(0)|^{p-1}|Db(0)-Db(X^0_t(x))||\zZ_t(0)|\ud t\\
&\lqq -p\delta |\yY^x_t-\zZ_t(0)|^{p}\ud t+ p|\yY^x_t-\zZ_t(0)|^{p-1}C(|x|)|X^0_t(x)||\zZ_t(0)|\ud t,
\end{align*}
where $C(|x|)=\max\limits_{|u|\lqq |x|+1}|D^2b(u)|$. 
Taking expectation, using the monotonicity of the integrals and Fubini's theorem imply
\begin{align*}
\ud\, &\EE[|\yY^x_t-\zZ_t(0)|^{p}]\lqq -p\delta \EE[|\yY^x_t-\zZ_t(0)|^{p}]\ud t+ pC(|x|)|X^0_t(x)|\EE[|\yY^x_t-\zZ_t(0)|^{p-1}\cdot|\zZ_t(0)|]\ud t.
\end{align*}
By Young's inequality and $|X^0_t(x)|\lqq e^{-\delta t}|x|$ for any $t\gqq 0$ and $x\in \RR^d$ it follows
\begin{align*}
\ud\, \EE[|\yY^x_t-\zZ_t(0)|^{p}]&\lqq -p\delta \EE[|\yY^x_t-\zZ_t(0)|^{p}]\ud t 
+ pC(|x|)|x| e^{-\delta t}\left(\EE[|\yY^x_t-\zZ_t(0)|^{p}]\ud t +
\EE[|\zZ_t(0)|^{p}]\right)\ud t.
\end{align*}
A straightforward calculation yields (for any  $p>0$) that there exist functions $P_1(t)$ and $P_2(t)$ of polynomial order 
(depending of $p$, $\delta$, $|x|$) such that 
\begin{equation}\label{eq:polygrows}
\EE[|\zZ_t(0)|^{p}] \lqq P_1(t)\quad \textrm{ and }
\quad \EE[|\yY^x_t|^{p}]\lqq P_2(t)\quad \textrm{ for any } t\gqq 0.
\end{equation}
Therefore,
\begin{align*}
\ud\, \EE[|\yY^x_t-\zZ_t(0)|^{p}]&\lqq -p\delta \EE[|\yY^x_t-\zZ_t(0)|^{p}]\ud t+ p2^{p}C(|x|)|x|e^{-\delta t}(P_1(t) + P_2(t))\ud t.
\end{align*}
The integral version of the Gr\"onwall inequality with negative linearity given 
in Lemma~1 in \cite{Mikami} yields 
\begin{align*}
\EE[|\yY^x_t-\zZ_t(0)|^{p}]
&\lqq p2^{p}C(|x|)|x|
e^{-p\delta t}\int_{0}^{t}
e^{p\delta s}e^{-\delta s} (P_1(s) + P_2(s))\ud s\\
&\lqq \frac{p2^{p}C(|x|)|x|}{\delta(p-1)} \max\limits_{0\lqq s\lqq t}\{P_1(s),P_2(s)\}e^{-\delta t}.
\end{align*}
Therefore, 
\begin{equation}\label{eq: nonhomOUtherm}
\lim\limits_{t\ra \infty}\Wp(\yY^x_t, \zZ_t(0)) \lqq \lim\limits_{t\ra \infty}\EE[|\yY^x_t-\zZ_t(0)|^{p}]=0.
\end{equation}
Combining \eqref{eq: OUthermalizationW2} and \eqref{eq: nonhomOUtherm} in \eqref{eq: triangulo} we conclude \eqref{eq: inhomOUthermState}. 

We continue with the case $p \in (0, p_*\wedge 1]$. Note that the case $p_*>1$ and $p\in (0,1]$ is also covered in the sequel. By Lemma~\ref{lem: Siorpaes} there exists 
a positive constant $C$ such that for the choice $\gamma=p$, $\vartheta = e^{\frac{\delta}{2} t}$ and any $t\gqq 0$ 
it follows 
\begin{equation}\label{eq:siorpaesnueva}
\PP(\dD_t^0)\lqq Ct e^{-\frac{\delta p}{2} t}, \qquad \mbox{ where we recall }\quad \dD_t^0=\Big\{\sup\limits_{0\lqq s\lqq t}|\zZ_s(0)|> \vartheta\Big\}.
\end{equation}
We split
\begin{align*}
\EE[|\yY^x_t-\zZ_t(0)|^{p}]=\EE[|\yY^x_t-\zZ_t(0)|^{p}\ind((\dD^0_t)^{\cc})]+
\EE[|\yY^x_t-\zZ_t(0)|^{p}\ind(\dD^0_t)]
=:I_1+I_2.
\end{align*}
We start with the term $I_1$.
The chain rule for $|x|^{p}_c=(\sqrt{|x|^2+c^2})^{p}$
and Hypothesis~\ref{hyp:dissipa9} yield
\begin{align*}
\ud\,|\yY^x_t-\zZ_t(0)|^{p}_{c}&=-p|\yY^x_t-\zZ_t(0)|^{p-2}_c\<\yY^x_t-\zZ_t(0),Db(X^0_t(x))(\yY^x_t-\zZ_t(0))\>\ud t\\
&\qquad +p|\yY^x_t-\zZ_t|^{p-2}_c\<\yY^x_t-\zZ_t(0),(Db(0)-Db(X^0_t(x)))\zZ_t(0)\>\ud t\\
&\lqq -p\delta|\yY^x_t-\zZ_t(0)|^{p-2}_c|\yY^x_t-\zZ_t(0)|^2 \ud t
+p|\yY^x_t-\zZ_t|^{p-1}_c C(|x|)|X^0_t(x)||\zZ_t(0)|\ud t\\
&= -p\delta|\yY^x_t-\zZ_t(0)|^{p}_c \ud t+p\delta c^{p}\ud t
+p c^{p-1} C(|x|)|X^0_t(x)||\zZ_t(0)|\ud t,
\end{align*}
where $C(|x|)=\max\limits_{|u|\lqq |x|+1}|D^2b(u)|$.  
On the event $(\dD^0_t)^{\cc}$ we have
\begin{align*}
\ud\,|\yY^x_t-\zZ_t(0)|^{p}_{c}\lqq -p\delta|\yY^x_t-\zZ_t(0)|^{p}_c \ud t+p\delta c^{p}\ud t
+pc^{p-1} C(|x|)|x|  e^{-(\delta/2) t}\ud t
\end{align*}
due to $|X^0_t(x)|\lqq e^{-\delta t}|x|$ for all $t\gqq 0$ and $x\in \RR^d$.
Hence
\begin{align*}
\ud\, \EE[|\yY^x_t-\zZ_t(0)|^{p}_{c}\ind((\dD^0_t)^{\cc})] &\lqq -p\delta\EE[|\yY^x_t-\zZ_t(0)|^{p}_c\ind((\dD^0_t)^{\cc})]\ud t +p\delta c^{p}\ud t
+pc^{p-1} C(|x|)|x|  e^{-(\delta/2) t}\ud t.
\end{align*}
The Gr\"onwall inequality in \cite{Mikami} implies
\begin{align*}
\EE[|\yY^x_t-\zZ_t(0)|^{p}\ind((\dD^0_t)^{\cc})]& \lqq 
\EE[|\yY^x_t-\zZ_t(0)|^{p}_{c}\ind((\dD^0_t)^{\cc})] \lqq 
c^{p}
+pc^{p-1} C(|x|)|x|e^{-p\delta t}\int_{0}^{t}e^{p\delta s} e^{-(\delta/2) s}\ud s.
\end{align*}
Then 
\[
\limsup_{t\ra \infty}
\EE[|\yY^x_t-\zZ_t(0)|^{p}\ind((\dD^0_t)^{\cc})] 
\lqq c^{p}\quad \textrm{ for all } c>0,
\]
which yields $\lim_{t\ra \infty}
\EE[|\yY^x_t-\zZ_t(0)|^{p}\ind((\dD^0_t)^{\cc})]=0$.

\noindent We continue with the term $I_2$.
By the H\"older inequality 
for  the index $p'/p$ where 
$p'=(p+p_*)/2$ and $r$ the conjugate index of $p'/p$ we have
\begin{align}
\EE[|\yY^x_t-\zZ_t(0)|^{p}&\ind(D_t)]\lqq 
\EE[|\yY^x_t|^{p}\ind(D_t)]+\EE[|\zZ_t(0)|^{p}\ind(D_t)]\nonumber\\
&\lqq 
(\EE[|\yY^x_t|^{p'}])^{p/p'}(\PP(D_t))^{1/r}+(\EE[|\zZ_t(0)|^{p'}])^{p/p'}(\PP(D_t))^{1/r}.\label{eq:holder99}
\end{align} 
By \eqref{eq:polygrows} and \eqref{eq:siorpaesnueva} the right-hand side of
\eqref{eq:holder99} tends to zero as $t\ra \infty$. As a consequence we have 
$\Wp(\yY^x_t,\zZ_t(0))\lqq (\EE[|\yY^x_t-\zZ_t(0)|^{p}])^{1\wedge (1/p)}$ which tends to zero as $t\ra \infty$.
By \eqref{eq: OUthermalizationW2} and \eqref{eq: triangulo} we obtain 
\eqref{eq: inhomOUthermState}.
\end{proof}
\subsection{\textbf{Auxiliary moment estimates}} 
\begin{lem}\label{lem:p>2}
For any $2\lqq p< p_*$ (and $p = 2$ if $p_* = 2$) there is a function of at most polynomial order $R(t)$ as $t \ra \infty$ and $\e_0\in (0, 1]$ such that for any $t\gqq 0$ and $0 < \e < \e_0$ we have  
\[
\EE[|X^\e_t(x)-X^0_t(x)|^{p_*}] \lqq \e^{p} R(t). 
\]
\end{lem}
\begin{proof}
First note that for  $G(u) = |u|^{p_*},p_*\gqq 2$ 
we have 
\begin{align*}
\nabla G(u) = p_* |u|^{p_*-2} u = p_* (|u|^2)^{\frac{p_*-2}{2}} u, \mbox{ with }\pd_i G(u) = p_* |u|^{p_*-2} u_i,
\end{align*}
\begin{align*}
 \sum_{ij} \pd_i \pd_j G(u)&\lqq  p_* |u|^{p_*-4} \big(d |u|^{2}+\sum_{ij} \frac{(p_*-2)}{2} (u_j^2 + u_i^2)\big)= p_* (p_*-1 )d |u|^{p_*-4} |u|^{2}.
 \end{align*}
Recall the notation \eqref{eq:LevyIto} for $L$.
The It\^o formula for $\Theta^\e_t= X^\e_t(x)-X^0_t(x)$ yields
\begin{align*}
\ud |\Theta^\e_t|^{p_*} 
&= -p_* |\Theta^\e_t|^{p_*-2} \<\Theta^\e_t, b(X^{\e}_t(x))-b(X^0_t(x))\> \ud t 
+ p_* |\Theta^\e_t|^{p_*-2} \<\Theta^\e_t, \e \Sigma^{1/2} \ud B_t\>\\
&\qquad+\frac{\e^2}{2} \trace(\Sigma^{1/2} \mbox{Hess} G(\Theta^\e_t)(\Sigma^{1/2})^{*}) \ud t\\
&\qquad  + \int_{\RR^d} \big(|\Theta^\e_t+ \e  z|^{p_*} -|\Theta^\e_t|^{p_*}- 
p_* |\Theta^\e_t|^{p_*-2} \<\Theta^\e_t, \e  z\> \ind\{|z|\lqq 1\}\big) \nu(\ud z) \ud t \\
&\qquad + \int_{\RR^d} \big(|\Theta^\e_t+ \e  z|^{p_*} -|\Theta^\e_t|^{p_*}\big) \ti N(\ud t, \ud z). 
\end{align*}
Taking expectation yields  
\begin{align*}
\EE[|\Theta^\e_t|^{p_*}]
&\lqq  - \delta p_* \int_0^t \EE\Big[|\Theta^\e_s|^{p_*}\Big] \ud s +\e^2 p_* (p_*-1 )d \trace(\Sigma^{1/2}(\Sigma^{1/2})^{*})\int_0^t \EE\Big[ |\Theta^\e_s|^{p_*-2} \Big] \ud s\\
&\qquad  + \int_0^t\int_{\RR^d} \EE\Big[|\Theta^\e_t+ \e  z|^{p_*} -|\Theta^\e_t|^{p_*}- 
p_*|\Theta^\e_t|^{p_*-2} \<\Theta^\e_t, \e  z\>\Big] \nu(\ud z) \ud s.
\end{align*}
By the mean value theorem we have 
\begin{align*}
\EE\Big[|\Theta^\e_t+  \e z|^{p_*} -|\Theta^\e_t|^{p_*}- 
p_* |\Theta^\e_t|^{p_*-2} \<\Theta^\e_t, \e z\>\Big]&\lqq \EE\Big[ p_* (p_*-1 )d \iint_{0}^1 |\Theta^\e_t+ \theta \vartheta \e  z|^{p_*-2} 
\ud \theta \ud \vartheta\Big] |\e  z|^2 \\
&\lqq (1\vee 2^{p_*-2}) \EE\Big[ p_* (p_*-1 )d ( |\Theta^\e_t|^{p_*-2}+ |\e z|^{p_*-2}) \Big] |\e  z|^2 \\
&\lqq (1\vee 2^{p_*-2}) p_* (p_*-1 )d \EE\Big[ |\Theta^\e_t|^{p_*-2}\Big] \Big(|\e z|^2 + |\e z|^{p_*} \Big)
\end{align*}
and 
\begin{align*}
&\int_{\RR^d} \EE\Big[|\Theta^\e_t+ \e z|^{p_*} -|\Theta^\e_t|^{p_*}- 
p |\Theta^\e_t|^{p_*-2} \<\Theta^\e_t, \e z\>\Big] \nu(\ud z)\\
&\hspace{2cm}\lqq C_{p_*, d} \Big(\int_{\RR^d} | z|^2 \nu(dz) + \int_{\RR^d} | z|^{p_*} \nu(dz)\Big) \e^2 \EE\Big[ |\Theta^\e_t|^{p_*-2}\Big].
\end{align*}
Hence there is a positive constant $K$ such that
\begin{align}\label{eq:boot}
\EE[|\Theta^\e_t|^{p_*}]
&\lqq  - \delta p_* \int_0^t \EE\Big[|\Theta^\e_s|^{p_*}\Big] \ud s +\e^2 K\int_0^t \EE\Big[ |\Theta^\e_s|^{p_*-2} \Big] \ud s.
\end{align}
For $p_*=2$ we have directly 
$
\EE[|\Theta^\e_t|^{p_*}]\lqq  \e^2 K t.
$
For $p_*>2$ we continue in \eqref{eq:boot} with Young's inequality 
\begin{align*}
\EE[|\Theta^\e_t|^{p_*}]
&\lqq  - \delta p_*\int_0^t \EE\Big[|\Theta^\e_s|^{p_*}\Big] \ud s +\e^2 K\int_0^t \EE\Big[ |\Theta^\e_s|^{p_*-2} \Big] \\
&\lqq  - \delta p_*\int_0^t \EE\Big[|\Theta^\e_s|^{p_*}\Big] \ud s +\e^2 K\int_0^t \EE\Big[ |\Theta^\e_s|^{p_*} \Big] \ud s+ \e^2 Kt\nonumber\\
&\lqq  - (\nicefrac{\delta}{2}) p_*\int_0^t \EE\Big[|\Theta^\e_s|^{p_*}\Big] \ud s + \e^2 Kt\nonumber
\end{align*}
for $\e<(\frac{\delta p_*}{2K})^{1/2}$.
Gr\"onwall's lemma applied to the preceding estimate yields the a priori estimate
$
\EE[|\Theta^\e_t|^{p_*}] \lqq \e^2 K t^2 =: \e^2 R_0(t).
$
Inserting the a priori estimate in \eqref{eq:boot} and using the H\"older inequality for $p_*>2$ we obtain 
\begin{align*}
\EE[|\Theta^\e_t|^{p_*}]
&\lqq  - \delta p_*\int_0^t \EE\Big[|\Theta^\e_s|^{p_*}\Big] \ud s +\e^2 K\int_0^t \EE\Big[ |\Theta^\e_s|^{p_*-2} \Big] \ud s\\
&\lqq  - \delta p_*\int_0^t \EE\Big[|\Theta^\e_s|^{p_*}\Big] \ud s +\e^2 K\int_0^t \EE\Big[ |\Theta^\e_s|^{p_*} \Big]^\frac{p_*-2}{p_*}\ud s\\
&\lqq  \e^{2+2 \frac{p_*-2}{p_*}} K^{1+ \frac{p_*-2}{p_*}} 
\int_0^t s^{2\frac{p_*-2}{p_*}}\ud s=:
\e^{2+2 \frac{p_*-2}{p_*}} R_1(t).
\end{align*}
By induction we deduce after the $i$-th iterations of the bootstrap the estimate 
\begin{align*}
\EE[|\Theta^\e_t|^{p_*}] \lqq \e^{2 \sum_{j=0}^i (\frac{p_*-2}{p_*})^j} R_i(t)
\end{align*}
for a polynomial order function $R_i(t)$. 
Clearly, 
$\lim_{i\ra  \infty} 2 \sum_{j=0}^i \Big(\frac{p_*-2}{p_*}\Big)^j = p_*$
and therefore for any $0<p<p_*$ there is an iteration $i_0=i_0(p_*,p)$ such that we obtain
$
\EE[|\Theta^\e_t|^{p_*}] \lqq \e^{p} R_{i_0}(t) 
$.
This finishes the proof of the lemma.
\end{proof}
\begin{lem}\label{lem:p<2} 
Let $p_*>0$. Then for any $p\in (0,2\wedge p_*)$
there exists a positive constant $C_{p}$ 
such that for any $t\gqq 0$ and $\e>0$
 we have  
\[
\EE[|X^\e_t(x)-X^0_t(x)|^{p}] \lqq \e^{p}(1+C_{p}\cdot t). 
\]
\end{lem}
\begin{proof}
Without loss of generality let $p_*\in (0,2]$.
It\^o's formula yields for $\Theta^\e_t = X^{\e}_t(x) -X^0_t(x)$ and the function $G(z)=|z|^{p}_c$
\begin{align*}
\ud |\Theta^\e_t|_c^p 
&= -p |\Theta^\e_t|_c^{p-2} \<\Theta^\e_t, b(X^{\e}_t(x))-b(X^0_t(x))\> \ud t 
+ p |\Theta^\e_t|_c^{p-2} \<\Theta^\e_t, \e \Sigma^{1/2} \ud B_t\>\\
&\qquad+\frac{\e^2}{2} \trace(\Sigma^{1/2} \mathrm{Hess}G(\Theta^\e_t)(\Sigma^{1/2})^{*}) \ud t\\
&\qquad + \int_{\RR^d} \big(|\Theta^\e_t+ \e z|_c^p -|\Theta^\e_t|_c^p- 
p |\Theta^\e_t|_c^{p-2} \<\Theta^\e_t, \e  z\> \ind\{|z|\lqq 1\}\big) \nu(\ud z) \ud t \\
&\qquad + \int_{\RR^d} \big(|\Theta^\e_t+ \e  z|_c^p -|\Theta^\e_t|_c^p\big) \ti N(\ud t, \ud z). 
\end{align*}
Taking expectation and using Hypothesis~\ref{hyp:dissipa9} we have
\begin{align*}
\EE[|\Theta^\e_t|_c^p]
&\lqq c^p - p \delta \int_0^t \EE[|\Theta^\e_t|_c^{p-2}|\Theta^{\e}_t|^2] \ud s +  \e^2 \int_{0}^t \trace(\Sigma^{1/2}\mathrm{Hess}G(\Theta^\e_s)(\Sigma^{1/2})^{*})\ud s\\
&\qquad  + \int_0^t\int_{\RR^d} \EE\Big[|\Theta^\e_t+ \e z|_c^p -|\Theta^\e_t|_c^p- 
p |\Theta^\e_t|_c^{p-2} \<\Theta^\e_t, \e  z\>\ind\{|z|\lqq 1\}\Big] \nu(\ud z) \ud s.
\end{align*}
Since 
$|x|^2=|x|^2_c-c^2$, we obtain
\begin{align}
\EE[|\Theta^\e_t|_c^p]
&\lqq c^p - p \delta \int_0^t \EE[|\Theta^\e_t|_c^{p} ]\ud s +p\delta c^{p}t +  \e^2 c^{p-2}tC(p,d)\trace(\Sigma^{1/2}(\Sigma^{1/2})^{*})\nonumber\\
&\qquad  + \int_0^t\int_{\RR^d} \EE\Big[|\Theta^\e_t+ \e  z|_c^p -|\Theta^\e_t|_c^p- 
p |\Theta^\e_t|_c^{p-2} \<\Theta^\e_t, \e  z\>\ind\{|z|\lqq 1\}\Big] \nu(\ud z) \ud s.\label{eq:expectation}
\end{align}  
In the sequel we estimate the second order term for small increments with the help of \eqref{eq:hessianest} by
\begin{align}
\int_{0}^{t}&\int_{|z|\lqq 1}
\EE\Big[|\Theta^\e_s+ \e z|_c^p -|\Theta^\e_s|_c^p- 
p |\Theta^\e_t|_c^{p-2} \<\Theta^\e_s, \e z\>\Big] \nu(\ud z)\ud s \nonumber\\
&\hspace{2cm}
\lqq C(p,d)\e^2 c^{p-2} t \int_{|z|\lqq 1} | z|^2\nu(\ud z)=:K_1 \e^2 c^{p-2} t.\label{eq:smalljumps}
\end{align}
For the large increments, we use the mean value theorem and obtain
\begin{align*}
&\int_0^t\int_{|z|>1} \EE\Big[|\Theta^\e_s+ \e  z|_c^p -|\Theta^\e_s|_c^p\Big] \nu(\ud z)\ud s= p\e \int_0^t\int_{|z|>1}\int_{0}^{1} \EE[|\Theta^\e_s+\theta \e  z|^{p-1}_c ]| z|\ud \theta\nu(\ud z) \ud s.
\end{align*}
For $p\in (0,1]$, note that
$|x+y|^p_c\lqq |x|^p+|y|^p+c^p$ for all $x,y \in \RR^d$. Then we have for all $t\gqq 0$ 
\begin{align}\label{eq:proot9} 
\int_0^t \int_{|z|>1} \EE\Big[|\Theta^\e_s+ \e  z|_c^p -|\Theta^\e_s|_c^p\Big] \nu(\ud z)\ud s &\lqq 
\int_0^t\int_{|z|>1} (\e^{p}|  z|^p+c^p) \nu(\ud z)\ud s\nonumber\\
&=t\e^p \int_{|z|>1}|z|^p \nu(\ud z)+tc^p\nu(\{|z|>1\}).
\end{align}
For $p>1$, due to $|x+y|^{p-1}_c\lqq |x|^{p-1}+|y|^{p-1}+c^{p-1}$ for all $x,y\in \RR^d$, we split the intermediate value as follows
\begin{align} 
&p\e \int_0^t\int_{|z|>1}\int_{0}^{1} \EE[|\Theta^\e_s+\theta \e  z|^{p-1}_c ]| z|\ud \theta\nu(\ud z) \ud s\nonumber\\
&\lqq 
p\e \int_0^t\int_{|z|>1} \EE\Big[|\Theta^\e_s|^{p-1} \Big]| z|\ud \nu(\ud z) \ud s
+p\e \int_0^t\int_{|z|>1}
|\e  z|^{p-1} |z|\nu(\ud z) \ud s
+p\e c^{p-1}t \nu(\{|z|>1\})
\nonumber\\
&= 
p \EE\Big[\int_0^t \e|\Theta^\e_s|^{p-1} \ud s\Big]
\int_{|z|>1}  | z|\ud \nu(\ud z) 
+p\e^p t \int_{|z|>1}
| z|^{p} \nu(\ud z)
+p\e c^{p-1}t \nu(\{|z|>1\})
\nonumber\\
&\qquad \lqq 
tp(1/K_3)^p \e^p+\frac{p\delta}{2}\int_{0}^{t}\EE\Big[|\Theta^\e_s|^{p}_c\Big]\ud s+p\e^p t \int_{|z|>1}
| z|^{p} \nu(\ud z)+p\e c^{p-1}t \nu(\{|z|>1\}),\label{eq:bigjumps}
\end{align}
where we have used in the last line the following weighted Young inequality 
\begin{align*}
 \int_0^t K_2\e|\Theta^\e_s|^{p-1} \ud s &\lqq 
(1/K_3)^p t K^p_2\e^p+K_3^{p/(p-1)}\int_{0}^{t}|\Theta^\e_s|^{p} \ud s\lqq  
t(1/K_3)^p \e^p+\frac{\delta}{2}\int_{0}^{t}|\Theta^\e_s|^{p}_c \ud s
\end{align*}
with 
$K_2=\int_{|z|>1}  | z|\ud \nu(\ud z)+1$ and
 $K_3=(\nicefrac{\delta}{2})^{p/(p-1)}$ followed by $|y|\lqq |y|_c$. 
Combining \eqref{eq:smalljumps} with \eqref{eq:bigjumps} for $p\gqq 1$, and 
\eqref{eq:proot9} with \eqref{eq:bigjumps} for $p< 1$, respectively,
 in \eqref{eq:expectation} we obtain
\begin{align*}
\EE[|\Theta^\e_t|_c^p]
&\lqq c^p - \frac{p \delta}{2} \int_0^t \EE[|\Theta^\e_t|_c^{p} ]\ud s +p\delta c^{p}t + K_0\e^2 c^{p-2}t+K_1\e^2 c^{p-2} t \\
&\qquad + tp(1/K_3)^p \e^p\cdot
\ind\{p\gqq 1\}
+p\e^p t \int_{|z|>1}
| z|^{p} \nu(\ud z)+p\e c^{p-1}t \nu(\{|z|>1\}),
\end{align*}
where $K_0=C(p,d)\trace(\Sigma^{1/2}(\Sigma^{1/2})^{*})$.
Since $|x|^p\lqq |x|^p_c$, the choice $c=c_\e=\e$ yields for all $t\gqq 0$
$
\EE[|\Theta^\e_t|^p]\lqq \e^p (1+Ct)
$
for some constant $C=C(p,\delta)$.
This completes the proof of the lemma.
\end{proof}

\section*{\textbf{Acknowledgments}}
The authors would like to thank the anonymous referee for her/his valuable comments which has led to significant improvement of the manuscript.
The authors would like to thank Carlos Gustavo Tamm de Ara\'ujo Moreira (Gugu) at IMPA for clarifying comments on the Hartman-Grobman theorem. 

\section*{Declarations}
\subsection*{Funding} 
The research of GB has been supported by the Academy of Finland, via 
the Matter and Materials Profi4 University Profiling Action, 
an Academy project (project No. 339228)
and the Finnish Centre of Excellence in Randomness and STructures
(project No. 346306). GB also would like to express
his gratitude to University of Helsinki for all the facilities used along
the realization of this work. 
The research of MAH has been supported by the 
proyecto de la Convocatoria 2020-2021: ``Stochastic dynamics of systems perturbed with small Markovian noise with applications in biophysics, climatology and statistics'' of the Facultad de Ciencias at Universidad de los Andes. 
\subsection*{Availability of data and material}
Data sharing not applicable to this article as no datasets were generated or analyzed during the current study.
\subsection*{Conflict of interests} The authors declare that they have no conflict of interest.
\subsection*{Authors' contributions}
All authors have contributed equally to the paper.

\bibliographystyle{amsplain}

\end{document}